\documentclass[11pt,a4paper,fleqn]{article}

\usepackage{amsthm}
\usepackage{amssymb}
\usepackage{amstext}
\usepackage{amsmath}
\usepackage{url}
\usepackage[colorlinks]{hyperref}
\hypersetup{
linkcolor=blue,
citecolor=blue,
}
\usepackage[margin = 40pt,font=small,labelfont=bf]{caption}
\usepackage{arydshln}
\usepackage{pdfsync}
\usepackage[T1]{fontenc}
\usepackage{color}
\usepackage{algorithm}
\usepackage{algorithmic}
\usepackage{enumerate}
\usepackage{bbm}
\usepackage{ae,aecompl}
\usepackage{pdfpages}
\usepackage{graphicx}
\usepackage{epsfig}
\usepackage{subfigure}

\numberwithin{equation}{section}

\newcommand{\thefont}[2]{\fontsize{#1}{#2}\fontshape{n}\selectfont}
\def\argmin{\mathop{\rm arg \; min}\limits}%

\theoremstyle{plain}
\newtheorem{prop}{Proposition}[section]
\newtheorem{coro}{Corollary}[section]
\newtheorem{theo}{Theorem}[section]
\newtheorem{lem}{Lemma}[section]
\newtheorem{hyp}{Assumption}[section]

\theoremstyle{definition}
\newtheorem{defin}{Definition}[section]

\theoremstyle{remark}
\newtheorem{rem}{Remark}[section]

\newcommand{\bigO}{\ensuremath{\mathcal O}}

\newcommand{\ba}{\boldsymbol{a}}
\newcommand{\bb}{\boldsymbol{b}}
\newcommand{\be}{\boldsymbol{e}}

\newcommand{\bp}{\boldsymbol{p}}

\newcommand{\bfun}{\boldsymbol{f}}
\newcommand{\bF}{\boldsymbol{F}}

\newcommand{\bX}{\mbox{$\boldsymbol{X}$}}
\newcommand{\bI}{\mbox{$\boldsymbol{I}$}}

\def\var{\mathop{\rm Var}\nolimits}%

\newcommand{\bmu}{\boldsymbol{\mu}}
\newcommand{\bnu}{\boldsymbol{\nu}}

\newcommand{\seb}{\hat{\bnu}^{h}_{n,\ubar{p}}}
\newcommand{\neb}{\hat{\bnu}_{n,p}}
\newcommand{\nebpbar}{\hat{\bnu}_{n,\ubar{p}}}

\newcommand{\E}{{\mathbb E}}
\newcommand{\R}{{\mathbb R}}

\renewcommand{\P}{{\mathbb P}}
\newcommand{\BB}{{\mathcal B}}

\newcommand{\U}{{\mathcal U}}

\newcommand{\WS}{W_2(\Omega)}
\newcommand{\WSac}{W_2^{ac}(\Omega)}

\newcommand{\FF}{\ensuremath{\mathcal F}}
\newcommand{\GG}{\ensuremath{\mathcal G}}

\newcommand{\DD}{\ensuremath{\mathcal D}}
\newcommand{\HH}{\ensuremath{\mathcal H}}

\newcommand{\1}{\rlap{\thefont{10pt}{12pt}1}\kern.16em\rlap{\thefont{11pt}{13.2pt}1}\kern.4em}
\newcommand{\ubar}[1]{\text{\b{$#1$}}}

\title{Minimax convergence rate for estimating the Wasserstein barycenter of random measures on the real line}

\author{ J\'{e}r\'{e}mie Bigot$^{1}$,  Ra\'ul Gouet$^{2}$, Thierry Klein$^{3}$ \& Alfredo L\'{o}pez$^{4}$  \\
\\  Institut de Math\'ematiques de Bordeaux et CNRS  (UMR 5251)$^{1}$  \\ Universit\'e de Bordeaux  \vspace{0.1cm}  \\ Depto. de Ingenier\'{\i}a Matem\'{a}tica and CMM (CNRS, UMI 2807)$^{2}$ \\ Universidad de Chile \vspace{0.1cm}\\  ENAC- Ecole nationale de l'aviation civile\\ et Institut de Math\'ematiques de Toulouse et CNRS  (UMR 5219)$^{3}$ \\ Universit\'e de Toulouse  \vspace{0.1cm}  \\ CSIRO Chile International Centre of Excellence$^{4}$}

\date{\today}

\begin{document}

\maketitle

\thispagestyle{empty}

\begin{abstract}
This paper is focused on the statistical analysis of probability measures $\bnu_{1},\ldots,\bnu_{n}$ on $\R$ that can be viewed as independent realizations of an underlying stochastic process. We  consider the situation of practical importance where the random measures $\bnu_{i}$ are absolutely continuous with densities $\bfun_{i}$ that are not directly observable. In this case, instead of the densities, we have access to datasets of real random variables  $(X_{i,j})_{1 \leq i \leq n; \; 1 \leq j \leq p_{i} }$ organized in the form of $n$ experimental units, such that $X_{i,1},\ldots,X_{i,p_{i}}$ are iid observations sampled from a random measure  $\bnu_{i}$ for each $1 \leq i \leq n$. In this setting, we focus on first-order statistics methods for estimating, from such data, a meaningful structural mean measure. For the purpose of taking into account phase and amplitude variations in the observations, we argue that the notion of Wasserstein barycenter is a relevant tool. The main contribution of this paper is to characterize the rate of convergence of a (possibly smoothed) empirical Wasserstein barycenter towards its population counterpart in the asymptotic setting where both $n$ and $\min_{1 \leq i \leq n} p_{i}$ may go to infinity.  The optimality of this procedure is discussed from the minimax point of view with respect to the Wasserstein metric. We also highlight the connection between our approach and the curve registration problem in statistics. Some numerical experiments are used to illustrate the results of the paper on the convergence rate of empirical Wasserstein barycenters.
\end{abstract}



%\section*{Acknowledgments} 
%{\bf {\color{red} To be completed !}}

\section{Introduction}

In this paper, we are concerned with the statistical analysis of a set of absolutely continuous measures $\bnu_{1},\ldots,\bnu_{n}$ on the real line $\R$, with supports included  in (a possibly unbounded) interval $\Omega \subset \R$,  that can be viewed as independent copies of an underlying random measure $\bnu$. In this setting, it is of interest to define and estimate a mean measure $\nu_{0}$ of the random probability measure $\bnu$. The notion of mean or averaging depends on the metric that is chosen to compare elements in a given data set. In this work, we consider the Wasserstein metric $d_W$ associated to the quadratic cost for the comparison of probability measures and we define $\nu_{0}$ as the population Wasserstein barycenter of $\bnu$, given by
$$
\nu_{0} = \argmin_{\mu \in \WS} \E \left[ d_W^{2}(\bnu,\mu) \right],
$$
where the above expectation  is taken with respect to the distribution of $\bnu$, and $\WS$ denotes the space of probability measures with support included  in $\Omega$ and with finite second moment.  A Wasserstein barycenter  corresponds to the Fr\'echet mean \cite{fre} that is an extension of the usual Euclidean mean to non-linear metric spaces. Throughout the paper, the population mean measure $\nu_{0}$ is also referred to as the structural mean of $\bnu$, which is a terminology borrowed from curve registration (see   \cite{ZhangMuller} and references therein).

Data sets leading to  the analysis of absolutely continuous measures appear in various research fields. Examples can be found in neuroscience  \cite{Srivastava}, demographic and genomics studies \cite{MR2736564,ZhangMuller},  economics  \cite{MR1946423}, as well as in biomedical imaging  \cite{PetersenMuller}. Nevertheless, in such applications, one does not directly observe raw data in the form of absolutely continuous measures. Indeed, we generally only have  access  to random observations sampled from different distributions that represent independent subjects or experimental units.

Thus, we propose to study the  estimation of the structural mean measure $\nu_{0}$ (the population Wasserstein barycenter) from a data set consisting of independent real random variables  $(X_{i,j})_{1 \leq i \leq n; \; 1 \leq j \leq p_{i} }$ organized in the form of $n$ experimental units, such that (conditionally on $\bnu_{i}$) the random variables  $X_{i,1},\ldots,X_{i,p_{i}}$ are iid observations sampled from the measure $\bnu_{i}$ with density $\bfun_{i}$, where $p_{i}$ denotes the number of observations for the $i$-th subject or experimental unit. The main purpose of this paper is to propose nonparametric estimators of the structural mean measure $\nu_{0}$ and to characterize their rates of convergence with respect to the Wasserstein metric in the asymptotic setting, where both $n$ and $\min_{1 \leq i \leq n} p_{i}$ may go to infinity.

\subsection{Main contributions}

Two types of nonparametric estimators are considered in this paper. The first one is given by the empirical Wasserstein barycenter 
of the set of measures $\tilde{\bnu}_1,\ldots,\tilde{\bnu}_{n}$, with $\tilde{\bnu}_{i} = \frac{1}{p_i} \sum_{j=1}^{p_i} \delta_{X_{i,j}}$ for $1 \leq i \leq n$. This estimator will be referred to as the non-smoothed empirical Wasserstein barycenter.  Alternatively, since the unknown probability measures $\bnu_{i}$ are supposed to be absolutely continuous, a second estimator is based on a preliminary smoothing step which consists in using standard kernel smoothing to construct estimators $\hat{\bfun}_{i}$ of the unknown densities $\bfun_{i}$ for each $1 \leq i \leq n$. Then, an estimator of $\nu_{0}$ is obtained by taking the empirical Wasserstein barycenter of the measures $\hat{\nu}_{i},\ldots,\hat{\nu}_{n}$, with $\hat{\nu}_{i}(A) := \int_A \hat{f}_{i} (x) dx$, $A \subset \R$ measurable. We refer to this class of estimators as smoothed empirical Wasserstein barycenters whose smoothness depend on the choice of the bandwidths in the preliminary kernel smoothing step. 

The rates of convergence of both types of estimators are derived for their (squared) Wasserstein risks, defined as their expected (squared) Wassertein distances from $\nu_{0}$, and their optimality is discussed from the minimax point of view. Finally, some numerical experiments with simulated data are used to illustrate these results. 

\subsection{Related work in the literature}

The notion of barycenter in the Wasserstein space, for a finite set of $n$ probability measures supported on $\R^{d}$ (for any $d \geq 1$), has been recently introduced in \cite{agueh2011barycenters} where a detailed characterization of such barycenters  in terms of existence, uniqueness and regularity is given using arguments from duality and convex analysis. However, the convergence (as $n \to + \infty$) of such Wasserstein barycenters is not considered in that work.

In the one dimensional case ($d = 1$), computing the Wasserstein barycenter of a finite set of probability measures simply amounts to averaging (in the usual way) their quantile functions. In statistics, this approach has been referred to as quantile synchronization \cite{ZhangMuller}. In the presence of phase variability in the data, quantile synchronization is known to be an appropriate alternative to the usual Euclidean mean of densities to compute a structural mean density that is more consistent with the data. Various asymptotic properties of quantile synchronization are studied in \cite{ZhangMuller} in a statistical model and asymptotic setting similar to that of this paper with $\min_{1 \leq i \leq n} p_{i} \geq n$. However, other measures of risk than the one in this paper are considered in \cite{ZhangMuller}, but the optimality of the resulting convergence rates of quantile synchronization is not discussed. 

The results of this paper are very much connected with those in \cite{Pana15} where a new framework is developed for the registration of multiple point processes on the real line for the purpose of separating amplitude and phase variation in such data.  In \cite{Pana15}, consistent estimators of the structural mean of multiple point processes are obtained by the use of smoothed Wasserstein barycenters with an appropriate choice of kernel smoothing. Also, rates of convergence of such estimators are derived for the Wasserstein metric. The statistical analysis of  multiple point processes is very much connected to the study of repeated observations organized in samples from independent subjects or experimental units. Therefore, some of our results in this paper on smoothed empirical Wasserstein barycenters are built upon the work in \cite{Pana15}. Nevertheless, novel contributions include the derivation of an exact formula to compute the risk of non-smoothed Wasserstein barycenters in the case of samples of equal size, and new upper bounds on the rate of convergence of the Wasserstein risk of non-smoothed and smoothed empirical Wasserstein barycenters, together with a discussion of their optimality from the minimax point of view.    

The construction of consistent estimators of a population Wasserstein barycenter for semi-parametric models of random measures can also be found in \cite{BK16} and \cite{MR3338645}, together with a discussion on their connection to the well known curve registration problem in statistics \cite{ramli,wanggas}.

\subsection{Organization of the paper}

In Section \ref{sec:model}, we  first briefly explain why using statistics based on the Wasserstein metric is a relevant approach for the  analysis of a set of random measures in the presence of phase and amplitude variations in their densities. Then, we introduce a deformable model for the registration of probability measures that is appropriate to study the statistical properties of empirical Wasserstein barycenters. The two types of nonparametric estimators described above are finally introduced at the end of Section \ref{sec:model}. The convergence rates and the optimality of these estimators are studied in Section \ref{sec:convrate}. Some numerical experiments with simulated data are proposed in Section \ref{sec:num} to highlight the finite sample performances of these estimators. Section \ref{sec:conclusion} contains a discussion on the main contributions of this work and their potential extensions. The proofs of the main results are gathered in a technical Appendix. Finally, note that we use  bold symbols $\bfun, \bnu, \ldots$ to denote random objects (except real random variables). %We sometimes write $\bfun(\cdot) = f(\omega, \cdot)$ or  $\bnu(\cdot) = \nu(\omega,\cdot)$ where $\omega$ is an element of an abstract probability space and $f$, $\nu$ are measurable functions. 

\section{Wasserstein barycenters for the estimation of the structural mean in a deformable model of probability measures} \label{sec:model} 

\subsection{The need to account for phase and amplitude variations}

To estimate a mean measure from the data  $(X_{i,j})_{1 \leq i \leq n; \; 1 \leq j \leq p_{i} }$, a natural approach is the following one. In a first step, one uses the $X_{i,j}$'s to compute estimators  $\hat{\bfun}_1,\ldots,\hat{\bfun}_n$  (e.g.\ via kernel smoothing) of the unobserved density functions $\bfun_1,\ldots,\bfun_n$ of the measures $\bnu_{1},\ldots,\bnu_{n}$. Then, an estimator of a mean density  might be defined as the usual Euclidean mean $\bar{\bfun}_{n} = \frac{1}{n} \sum_{i=1}^{n}  \hat{\bfun}_{i}$, which is also classically referred to as the cross-sectional mean in curve registration. At the level of measures, it corresponds to computing the arithmetical mean measure  $\bar{\bnu}_{n} = \frac{1}{n} \sum_{i=1}^{n}  \hat{\bnu}_{i}$.  The Euclidean mean $\bar{\bfun}_{n}$ is to the Fr\'echet mean of the $\hat{\bfun}_i$'s with respect to the usual squared distance in the Hilbert space $L^{2}(\Omega)$ of square integrable functions on $\Omega$. Therefore, it only accounts for linear variations in amplitude in the data. However, as remarked in \cite{ZhangMuller}, in many applications, it is often of interest to also incorporate an analysis of phase variability (i.e.\ time warping) in such functional objects, since it may lead to a better understanding of the structure of the data. In such settings, the use of the standard squared distance in $L^{2}(\Omega)$ to compare density functions  ignores a possible significant source of phase variability in the data. 

To better account for phase variability in the data, it has been proposed in \cite{ZhangMuller} to introduce the so-called method of quantile synchronization as an alternative to the cross sectional mean $\bar{\bfun}_{n}$. It amounts  to computing the mean measure $\bnu_{n}^{\oplus}$ (and, if it exists, its density $\bfun_{n}^{\oplus}$) whose quantile function is
\begin{equation}
\bar{\bF}_{n}^{-} = \frac{1}{n} \sum_{i=1}^{n} \bF^{-}_{i}, \label{eq:quantoplus}
\end{equation}
where $\bF^{-}_{i}$ denotes the quantile function of the measure $\bnu_{i}$ with density $\bfun_{i}$.

The statistical analysis of quantile synchronization, as studied in \cite{ZhangMuller}, complements the quantile normalization method originally proposed in \cite{Bolstad2003} to align density curves in microarray data analysis. This method is therefore appropriate for the registration of density functions and the estimation of phase and amplitude variations as explained in details in \cite{Pana15}. 

Let us now assume that $\bnu_1,\ldots,\bnu_n$ are random elements taking values in the set of absolutely continuous measures contained in $\WS$.
In this setting, it can be checked (see e.g.\ Proposition \ref{prop:existence} below) that quantile synchronization corresponds to computing the empirical Wasserstein barycenter of the random measures $\bnu_1,\ldots,\bnu_n$, namely
$$
\bnu_{n}^{\oplus} = \argmin_{\mu \in \WS} \frac{1}{n} \sum_{i=1}^{n} d_W^{2}(\bnu_{i},\mu).
$$
Therefore, the notion of averaging by quantile synchronization corresponds to using the Wasserstein distance $d_W$ to compare probability measures, which leads to a notion of measure averaging that may better reflect the structure of the data than the arithmetical mean in the presence of phase and amplitude variability.

%%%%
%%%%

To illustrate the differences between using Euclidean and Wasserstein distances  to account for phase and amplitude variation, let us assume that the measures $\bnu_1,\ldots,\bnu_n$ have densities $\bfun_1,\ldots,\bfun_n$  obtained from the following location-scale model: we let $f_0$ be a density on $\R$ having a finite second moment and, for $(\ba_{i},\bb_{i})  \in (0,\infty) \times \R,\ i=1,\ldots,n$ a given sequence of independent random variables, we define 
\begin{equation} \label{eq:locscale}
\bfun_{i}(x) := \ba_i^{-1}f_{0}\left( \ba_i^{-1}(x-\bb_{i})\right), \; x \in \R, \; 1 \leq i \leq n.
\end{equation}
The  sources of variability of the densities from model \eqref{eq:locscale} are the variation in location along the $x$-axis, and the scaling variation. In Figure \ref{fig:exdata}(a), we plot a sample of $n=100$ densities from model $\eqref{eq:locscale}$ with $f_{0}$ being the standard Gaussian density, $\ba_{i} \sim \U([0.8,1.2])$ and $\bb_{i} \sim \U([-2,2])$, where $\U([x,y])$ denotes the uniform distribution on the interval $[x,y]$. In this numerical experiment, there is more variability in phase (i.e.\ location) than in amplitude (i.e.\ scaling), which can also be observed at the level of quantile functions as shown by Figure \ref{fig:exdata}(b).

\begin{figure}[htbp]
\centering
\subfigure[Densities $\bfun_1,\ldots,\bfun_n$ sampled from a location-scale model]{\includegraphics[width=7cm,height=5cm]{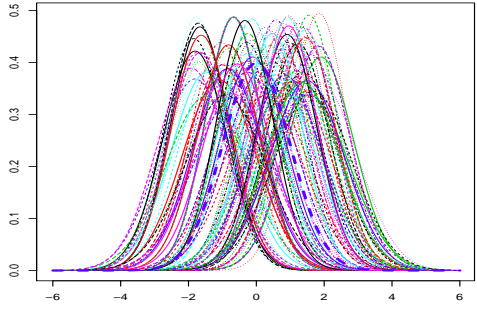}}
\subfigure[Quantile functions $\bF^{-}_{1},\ldots,\bF^{-}_{n}$  of $\bfun_1,\ldots,\bfun_n$]{\includegraphics[width=7cm,height=5cm]{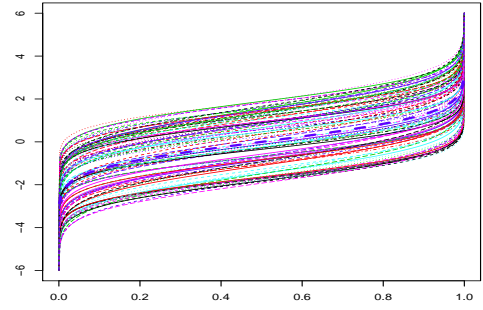}}

\subfigure[Euclidean mean density $\bar{\bfun}_{n}$]{\includegraphics[width=7cm,height=5cm]{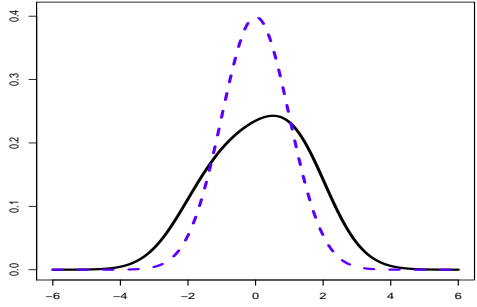}}
\subfigure[Quantile function of the arithmetical mean measure $\bar{\bnu}_{n}$ with density $\bar{\bfun}_{n}$]{\includegraphics[width=7cm,height=5cm]{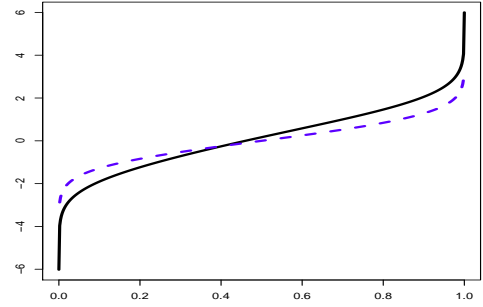}}

\subfigure[Density $\bfun_{n}^{\oplus}$ by quantile synchronization]{\includegraphics[width=7cm,height=4cm]{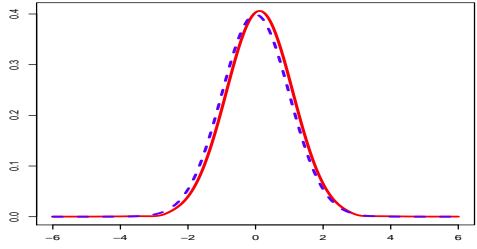}}
\subfigure[Quantile function of the Wasserstein barycenter $\bnu_{n}^{\oplus}$ with density $\bfun_{n}^{\oplus}$]{\includegraphics[width=7cm,height=4cm]{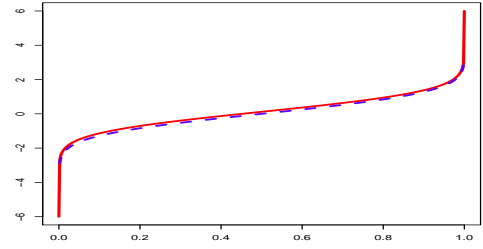}}

\caption{An example of $n=100$ random densities (a) with quantile functions (b) sampled from the location-scale model \eqref{eq:locscale} with $f_{0}$ the standard Gaussian density, $\ba_{i} \sim \U([0.8,1.2])$ and $\bb_{i} \sim \U([-2,2])$. (c,d) The solid-black curves are the Euclidean mean $\bar{\bfun}_{n}$ and its quantile function. (e,f) The solid-red curves are the structural mean $\bfun_{n}^{\oplus}$ given by quantile synchronization and the quantile function  of the empirical Wasserstein barycenter  $\bnu_{n}^{\oplus}$. In all the figures, the dashed-blue curves are either the density $f_{0}$ or its quantile function in the location-scale model \eqref{eq:locscale}.} \label{fig:exdata}
\end{figure}

 In the location-scale model \eqref{eq:locscale}, it can be checked, e.g.\ using the quantile averaging formula \eqref{eq:quantoplus}, that the empirical Wasserstein barycenter $\bnu_{n}^{\oplus}$ is the probability measure with density 
$$
\bfun_{n}^{\oplus}(x) = \bar{\ba}_{n}f_{0}\left( \bar{\ba}_{n}^{-1}(x-\bar{\bb}_{n})\right),
$$
where $\bar{\ba}_{n} = \frac{1}{n} \sum_{i=1}^{n} \ba_{i}$ and $\bar{\bb}_{n} = \frac{1}{n} \sum_{i=1}^{n} \bb_{i}$. Hence, if we assume that $\E(\ba_{1}) = 1$ and $\E(\bb_{1}) = 0$, it follows that $d_W^{2}(\bnu_{n}^{\oplus},\nu_{0})$ converges almost surely to $0$ as  $n \to \infty$, meaning that $\bnu_{n}^{\oplus}$ is a consistent estimator of $\nu_{0}$ as shown by Figure \ref{fig:exdata}(f). On the contrary, the arithmetical mean measure $\bar{\bnu}_{n}$ is clearly not a consistent estimator of  $\nu_{0}$,  as it can be observed in Figure \ref{fig:exdata}(d). \\

\begin{rem} 
It is  clear that, in the above location-scale model, one may easily prove that $\bfun_{n}^{\oplus}$ converges almost surely to $f_{0}$ as $n \to \infty$ for various distances between density functions as illustrated by Figure \ref{fig:exdata}(e). However, in this paper, we restrict our attention to the problem of how the structural mean measure $\nu_{0}$ can be estimated from empirical Wasserstein barycenters with respect to the Wasserstein distance between probability measures. Showing that the density (if it it exists)  of such estimators converges to the density $f_{0}$ of $\nu_{0}$ is not considered in this work.
\end{rem} 

%%%
%%%%%

\subsection{Barycenters in the Wasserstein space}

Let $\Omega$ be an interval  of $\R$, that is possibly unbounded. We let $\WS$ be the set of probability measures over $(\Omega,{\cal B}(\Omega))$, with finite second moment, where ${\cal B}(\Omega)$ is the $\sigma$-algebra of Borel subsets of $\Omega$. We also denote by $\WSac$ the set of measures $\nu \in \WS$ that are absolutely continuous with respect to the Lebesgue measure $dx$ on $\R$.  The cumulative distribution function (cdf) and the quantile function of $\nu$ are denoted respectively by $F_\nu$ and $F_\nu^{-}$.

\begin{defin}  \label{defi:wdist}
The quadratic Wasserstein distance $d_W$ in $\WS$ is defined by
\begin{equation} \label{eq:wdist}
d_W^2(\mu,\nu) := \int_{0}^1 (F_{\mu}^{-}(\alpha) - F_{\nu}^{-}(\alpha))^2 d\alpha, \mbox{ for any } \mu, \nu \in\WS.
\end{equation}
\end{defin}

It can be shown that $\WS$ endowed with $d_W$ is a metric space, usually called Wasserstein space. For a detailed analysis of $\WS$ and its connection with optimal transport theory, we refer to \cite{villani-topics}.  A probability measure $\bnu$ in $\WS$ is said to be random if it is sampled from a distribution $\P$ on $(\WS, \BB \left( \WS \right)$, where $\BB \left( \WS \right)$ is the  Borel $\sigma$-algebra  generated by the topology induced by the distance $d_W$.

\begin{defin}[Square-integrability] \label{def:square}
The random measure $\bnu$ is said to be square-integrable if 
$$
\E(d_W^2(\mu,\bnu))=\int_{\WS}d_W^2(\mu,\nu)d\P(\nu)< + \infty
$$
for some (thus for every) $\mu \in \WS$.
\end{defin}

\begin{defin}[Population and empirical Wasserstein barycenters]  \label{defi:barW2}
Let $\bnu$ be a $\WS$-valued square integrable random measure with distribution $\P$. The population Wasserstein barycenter of $\bnu$ is defined as the minimizer of
$$
\mu \mapsto \int_{\WS}d_W^2(\mu,\nu)d\P(\nu)  \mbox{ over } \mu \in \WS.
$$
The empirical Wasserstein barycenter of  $\nu_{1},\ldots,\nu_{n}\in \WS$ is defined as the minimizer of
$$
\mu \mapsto \frac{1}{n} \sum_{i=1}^{n}  d_W^2(\mu,\nu_{i})  \mbox{ over } \mu \in \WS.
$$

\begin{rem} 
In the whole paper, we assume that the model of random probability measure is well defined in the sense that all applications from an abstract probability space  to $\WS$ are measurable and hence we can apply Fubini's theorem. For an example of a rigorous model satisfying this kind of assumptions we refer to \cite{BK16}.
\end{rem}

\end{defin}
\begin{prop}\label{prop:existence} Let $\bnu\in\WS$ be a square-integrable random measure then 
\begin{enumerate}
\item[(i)] The  exists a unique barycenter $\nu_0$ of $\bnu$.
\item[(ii)] $F^-_{\nu_0}=\E\left[F^-_{\bnu}\right]$.
\item[(iii)] $\var(\bnu):=\E\left[ d^2_W\left(\bnu,\nu_0\right)\right]=\int_0^1\var\left( F^-_{\bnu}(\alpha) \right)d\alpha$.
\end{enumerate}
\end{prop}
\begin{proof}
Points  $(i)$ and  $(ii) $ are consequences of Proposition 4.1 in \cite{BGKL15}. Let us prove $(iii)$. From \eqref{eq:wdist} and Fubini's theorem, we have
$$
\var\left(\bnu \right) =\E\left[ \int_0^1\left( F^-_{\bnu}(\alpha)-F^-_{\nu_0}(\alpha)\right)^2d\alpha\right]=\int_0^1\E\left[\left( F^-_{\bnu}(\alpha)-F^-_{\nu_0}(\alpha)\right)^2\right]d\alpha =\int_0^1\var\left(F^-_{\bnu}(\alpha) \right)d\alpha.
$$
\end{proof}

\subsection{A deformable model of probability measures} \label{sec:deform}
Let $\bnu$ be a $\WS$-valued square integrable random probability measure. We use the notation $\bF$ and $\bF^{-}$ to denote the cumulative distribution function (cdf) and the quantile function of the random measure $\bnu$. Let us also denote by $\nu_0$  the barycenter of $\bnu$ (the existence and unicity of $\nu_0$ is ensured thanks to Proposition \ref{prop:existence}) and by  $\bnu_1, \ldots,\bnu_n$  independent copies of $\bnu$. In this paper, we consider a  deformable model of random probability measures satisfying the following assumptions:

\begin{hyp}\label{A1}
$\bnu\in\WSac$, a.s.\ and is a square integrable random probability measure in the sense of Definition \ref{def:square}.
\end{hyp}

\begin{hyp}\label{A2}
$\nu_0\in\WSac$.
\end{hyp}

\begin{hyp}\label{A3}
For each $1 \leq i \leq n$,  conditionally on  $\bnu_i$, the observations $X_{i,1},\ldots,X_{i,p_{i}}$ are iid random variables sampled from  $\bnu_i$, where $p_{i} \geq 1$ is a known integer.
\end{hyp}
\begin{rem}
Since $\bnu$ is square integrable, it follows from  Proposition \ref{prop:existence} that
\begin{equation}\label{eq:A3}
\E \left[ \bF^{-}(\alpha)   \right] = F_{0}^{-}(\alpha), \quad  \mbox{for all } \alpha \in ]0,1[, \quad \mbox{ and } \quad  0 \leq \int_{0}^{1}  \var\left(  \bF^{-}  (\alpha) \right)d\alpha < + \infty.
\end{equation}
It should be remarked that  similar assumptions are considered in \cite{Pana15} to characterize a population barycenter in $\WS$  for the purpose of estimating phase and amplitude variations from the observations of multiple point processes. For   examples of parametric models satisfying the Assumptions \ref{A1}-\ref{A3}, we refer to \cite{BK16} and \cite{MR3338645}. The main restriction of this deformable model  is that $\nu_0$ is assumed to be absolutely continuous.
\end{rem}

\subsection{Non-smoothed empirical barycenter} \label{sec:alternative}

To estimate the structural mean measure $\nu_{0}$ from the data $(X_{i,j})_{1 \leq i \leq n; \; 1 \leq j \leq p_{i} }$, a first approach consists in computing straightaway the barycenter of the empirical measure $\tilde{\bnu}_{1},\ldots, \tilde{\bnu}_{n}$ where $\ubar{p}=\left(p_1,\ldots,p_n\right)$, and $\tilde{\bnu}_{i} =\frac{1}{p_i}\sum_{j=1}^{p_i}\delta_{X_{i,j}}$ ($\delta_a$ denotes the Dirac mass at point $x\in\Omega$). The non-smoothed empirical barycenter is thus defined as 
\begin{equation}
\nebpbar = \argmin_{\mu \in \WS} \frac{1}{n} \sum_{i=1}^{n} d_W^{2}(\tilde{\bnu}_{i},\mu). \label{hatbnunp}
\end{equation}
In the case where $p_1 = p_2 = \ldots = p_n = p$, we have the following procedure for computing the non-smoothed empirical barycenter.
 For each $1 \leq i \leq n$, we denote by $X_{i,1}^{\ast} \leq X_{i,2}^{\ast} \leq \ldots \leq X_{i,p}^{\ast} $ the order statistics corresponding to the $i$-th sample of observations $(X_{i,j})_{1 \leq j \leq p }$, and we define
$$
\bar{X}_{j}^{\ast} = \frac{1}{n} \sum_{i=1}^{n} X_{i,j}^{\ast}, \mbox{ for all } 1 \leq j \leq p.
$$
Thanks to Proposition \ref{prop:existence}, the quantile function of the empirical Wasserstein barycenter is the average of the quantile functions of  $\tilde{\bnu}_{1},\ldots, \tilde{\bnu}_{n}$, and thus we obtain the formula 
\begin{equation}
\neb = \frac{1}{p}\sum_{j=1}^p\delta_{\bar{X}_{j}^{\ast}}. \label{eq:nonregestim}
\end{equation}
Note that we use the notation $\neb$ instead of $\nebpbar$ to denote the non-smoothed empirical barycenter  in the case $p_1 = p_2 = \ldots = p_n = p$.

\subsection{ Smoothed empirical barycenter}

An alternative approach is to use a smoothing step to obtain estimated densities and then compute the barycenter. In a first step, to obtain estimators $\hat{\bfun}^{h_1}_{1}, \ldots, \hat{\bfun}^{h_n}_{n}$ of $\bfun_{1},\ldots,\bfun_{n}$, we use kernel smoothing, where  $h_1,\ldots,h_n$ are positive  bandwidth parameters that may be different for each subject or experimental unit. In this paper, to analyze the convergence of smoothed empirical barycenter in $\WS$, we shall investigate a non-standard choice for the kernel function that has been proposed in  \cite{Pana15}. In Section \ref{sec:convrate}, we  give a precise definition of the resulting estimators based on the work in \cite{Pana15}. However, at this point, it is not necessary to go into such details. Then, in a second step, an estimator of $\nu_0$ is  given by $\seb$, with $\ubar{p}=\left(p_1,\ldots, p_n\right)$, defined as the measure whose quantile function is given by
\begin{equation}
\hat{\bF}_{h}^{-}(\alpha) = \frac{1}{n} \sum_{i=1}^{n} \hat{\bF}_{i}^{-}(\alpha), \; \alpha \in [0,1], \label{hatbfunh}
\end{equation}
where $\hat{\bF}_{i}^{-}$ denotes the quantile function of the density $\hat{\bfun}^{h_{i}}_{i}$ for each $1 \leq i \leq n$.  If for each $1 \leq i \leq n$, we denote by  $\hat{\bnu}_{i}^{h_{i}}$ the measure with density $\hat{\bfun}_{i}^{h_{i}}$, then by Proposition \ref{prop:existence}, one has that $\seb$ is also characterized as the following smoothed empirical Wasserstein barycenter
\begin{equation}
\seb = \argmin_{\mu \in \WS} \frac{1}{n} \sum_{i=1}^{n} d_W^{2}(\hat{\bnu}_{i}^{h_{i}},\mu). \label{eq:charahatnu}
\end{equation}

\section{Convergence rate for estimators of the population Wasserstein barycenter} \label{sec:convrate}

In this section, we discuss the rates of convergence of the estimators $\nebpbar$ and  $\seb$, that are respectively characterized by equations \eqref{hatbnunp} and \eqref{eq:charahatnu}.  Some of the results presented below are using the work in \cite{W1} on a detailed study of the variety of rates of convergence of an empirical measure on the real line toward its population counterpart in the Wasserstein metric. Then, we discuss the optimality of these estimators  from the minimax point of view following the  guidelines in nonparametric statistics to derive optimal rates of convergence  (see e.g.\ \cite{MR2724359} for an introduction to this topic).

\subsection{Non-smoothed empirical barycenter in the case of samples of equal size} \label{sec:convratepequal}

Let us first characterize the rate of convergence of $\neb$, in the specific case where  samples of observations per unit are  of equal size, namely when
$
p_1 = p_2 = \ldots = p_n = p.
$
In what follows, we let $Y_1,\ldots,Y_p$ be iid random variables sampled from the population mean measure $\nu_{0}$ (independently of the data $(X_{i,j})_{1 \leq i \leq n; \; 1 \leq j \leq p }$), and we denote by $\bmu_{p} = \frac{1}{p}\sum_{k=1}^{p} \delta_{Y_{k}}$ the corresponding empirical measure. 

\begin{theo} \label{theo:ratehatbnu}
If  Assumptions  \ref{A1}, \ref{A2} and \ref{A3} are satisfied and if  $p_1 = p_2 = \ldots = p_n = p$, then the estimator  $\neb$ satisfies 
\begin{eqnarray}
\E \left[ d_W^{2}(\neb, \nu_{0}) \right] & = & \frac{1}{n}  \int_{0}^{1}  \var\left(  \bF^{-}  (\alpha) \right)d\alpha +  \frac{1}{p n} \sum_{j=1}^{p}  \var \left( Y_{j}^{\ast}  \right) + \sum_{j=1}^{p} \int_{(j-1)/p}^{j/p}   \left( \E \left[ Y_{j}^{\ast}  \right]- F_{0}^{-}(\alpha) \right)^2 d\alpha, \nonumber \\
& = & \frac{1}{n}  \int_{0}^{1}  \var\left(  \bF^{-}  (\alpha) \right)d\alpha +  \frac{1-n}{p n} \sum_{j=1}^{p}  \var \left( Y_{j}^{\ast}  \right) + \E \left[ d_W^{2}(\bmu_{p}, \nu_{0}) \right], \label{eq:exactrate} 
\end{eqnarray}
where $Y_1^{\ast} \leq Y_2^{\ast}  \leq \ldots \leq Y_p^{\ast}$ denote the order statistics of the sample $Y_1,\ldots,Y_p$.
\end{theo}

Theorem \ref{theo:ratehatbnu} provides exact formulas to compute the rate of convergence (for the expected squared Wasserstein distance) of $\neb$. Formula \eqref{eq:exactrate} relies on the computation of the variances of the order statistics of  iid variables  $Y_1,\ldots,Y_p$ sampled from the population mean measure $\nu_{0}$, and on the computation of the rate of convergence of $ \E \left[ d_W^{2}(\bmu_{p}, \nu_{0}) \right]$. We discuss below some examples where equality \eqref{eq:exactrate} may be used to derive a sharp rate of convergence for  $\neb$. \\

\noindent {\bf The case where $\nu_{0}$ is the uniform distribution on $[0,1]$.} In this setting,  it is known (see  e.g.\ Section 4.2 in   \cite{W1}) that
$$
 \var \left( Y_{j}^{\ast}  \right)  = \frac{j(p-j+1)}{(p+1)^2(p+2)} \mbox{ and thus } \sum_{j=1}^{p}  \var \left( Y_{j}^{\ast}  \right)  = \frac{p}{6(p+1)}.
$$ 
Moreover,  from Theorem 4.7 in \cite{W1}, it follows that  $\E \left[ d_W^{2}(\bmu_{p}, \nu_{0}) \right] = \frac{1}{6 p}$. Therefore, thanks to equality \eqref{eq:exactrate},  we obtain that
\begin{eqnarray}
\E \left[ d_W^{2}(\neb, \nu_{0}) \right] 
& = &  \frac{1}{n}  \int_{0}^{1}  \var\left(  \bF^{-}  (\alpha) \right)d\alpha +  \frac{1-n}{6 n(p+1)} + \frac{1}{6 p} \nonumber \\
& = &  \frac{1}{n}  \int_{0}^{1}  \var\left(  \bF^{-}  (\alpha) \right)d\alpha  + \frac{1}{6} \left(  \frac{1}{n(p+1)} + \frac{1}{p(p+1)} \right). \label{eq:rateunif}
\end{eqnarray}
Equality \eqref{eq:rateunif} thus shows that, when  $\nu_{0}$ is the uniform distribution on $[0,1]$, the rate of convergence of $\neb$ is of the order
\begin{equation}
\E \left[ d_W^{2}(\neb, \nu_{0}) \right]  \asymp \frac{1}{n} + \frac{1}{np} + \frac{1}{p^2},  \label{eq:ratebarunif}
\end{equation}
and that this rate is sharp. \\

\noindent {\bf The case where $\nu_{0}$ is the one-sided exponential distribution.} From Theorem 4.3 in \cite{W1}, one has that
\begin{equation}
 \frac{1}{2p} \sum_{j=1}^{p}  \var \left( Y_{j}^{\ast}  \right) \leq \E \left[ d_W^{2}(\bmu_{p}, \nu_{0}) \right]  \leq \frac{2}{p} \sum_{j=1}^{p}  \var \left( Y_{j}^{\ast}  \right), \label{eq:boundvartwosided}
\end{equation} 
for any distribution $\nu_{0} \in \WS$. Therefore, combining the above inequalities with \eqref{eq:exactrate}, it follows that
\begin{equation}
 \E \left[ d_W^{2}(\neb, \nu_{0}) \right]  \leq \frac{1}{n}  \int_{0}^{1}  \var\left(  \bF^{-}  (\alpha) \right)d\alpha +  \frac{1+n}{p n} \sum_{j=1}^{p}  \var \left( Y_{j}^{\ast}  \right). \label{eq:boundvar} 
\end{equation} 
Now (using e.g.\ Remark 6.13 in \cite{W1}) one has that if $\nu_{0}$ is the one-sided exponential distribution (with density $e^{-x}$ for $x \geq 0$) then
$$
\sum_{j=1}^{p}  \var \left( Y_{j}^{\ast}  \right) = \sum_{j=1}^{p} \frac{1}{j} \sim \log(p) \mbox{ as } p \to + \infty.
$$
Therefore, there exist a constant $c > 0$ such that
\begin{equation}
 \E \left[ d_W^{2}(\neb, \nu_{0}) \right]  \leq  \frac{1}{n}  \int_{0}^{1}  \var\left(  \bF^{-}  (\alpha) \right)d\alpha  +  c \left(1 + \frac{1}{n}\right) \frac{\log(p)}{p}  \label{eq:ratebarexp}
\end{equation}
for all sufficiently large $p$. Hence, when $\nu_{0}$ is the exponential distribution the above inequalities show that the  rate of convergence of $\neb$ is of the order $\bigO\left( \frac{1}{n} +  \log(p)\left( \frac{1}{np}  + \frac{1}{p} \right) \right)$. \\

\noindent {\bf The case where $\nu_{0}$ is the standard Gaussian distribution.} Deriving a sharp rate of convergence for $\neb$  using inequalities \eqref{eq:exactrate} combined with \eqref{eq:boundvartwosided} requires computing the variances of the order statistics of  iid random variables. To the best of our knowledge, obtaining a sharp estimate for $\var \left( Y_{j}^{\ast}  \right)$ for any $1 \leq j \leq p$ remains a difficult task except for specific distributions. 
Nevertheless,  if $\nu_{0}$ is assumed to be a log-concave measure, then it is possible to use the results in Section 6 of \cite{W1} which provide sharp bounds  on the variances of order statistics for such probability measures. 

For example, if $\nu_{0}$ is the standard Gaussian distribution, then by Theorem 4.3 and Corollary 6.14 %Lemma 6.5 and the results in Section 6.5 in
in \cite{W1} we obtain that there exist two constants $c_1,c_2 > 0$ such that
$$
c_1 \frac{\log (\log(p)) }{p} \leq  \frac{1}{p} \sum_{j=1}^{p}  \var \left( Y_{j}^{\ast}  \right) \leq c_2 \frac{\log (\log(p)) }{p}.
$$
Therefore, combining the above upper bound with \eqref{eq:boundvar}, one finally has that 
\begin{equation}
 \E \left[ d_W^{2}(\neb, \nu_{0}) \right]  \leq \frac{1}{n}  \int_{0}^{1}  \var\left(  \bF^{-}  (\alpha) \right)d\alpha +  c_2 \left(\frac{1}{n} + 1 \right)  \frac{\log (\log(p)) }{p}. \label{eq:twosidedboundGaussian} 
\end{equation} 
when $\nu_{0}$ is the standard Gaussian distribution. In this setting, the rate of convergence is thus of the order $\bigO\left( \frac{1}{n} +  \log(\log(p)) \left( \frac{1}{np}  + \frac{1}{p} \right) \right)$. \\

\noindent {\bf Upper bounds in more general cases.}  If one is  interested in deriving an upper bound on $\E \left[ d_W^{2}(\neb, \nu_{0}) \right]$ for a larger class of measures $\nu_{0} \in \WS$ (e.g.\ beyond the log-concave case), another approach is as follows. Noting that the term $\frac{1-n}{p n} \sum_{j=1}^{p}  \var (Y_{j}^{\ast})$ in equality \eqref{eq:exactrate} is negative, a straightforward consequence of Theorem \ref{theo:ratehatbnu}   is the following upper bound
\begin{equation}
\E \left[ d_W^{2}(\neb, \nu_{0}) \right] \leq \frac{1}{n}  \int_{0}^{1}  \var\left(  \bF^{-}  (\alpha) \right)d\alpha + \E \left[ d_W^{2}(\bmu_{p}, \nu_{0}) \right].
\label{eq:upperboundrate}
\end{equation}

Then, thanks to inequality \eqref{eq:upperboundrate}, to derive the rate of convergence of $\neb$, it remains to control the rate of convergence of the empirical measure $\bmu_{p}$ to $\nu_{0}$ for the expected squared Wasserstein distance. This issue is discussed in detail in \cite{W1}. In particular, the work in \cite{W1} describes a variety of rates for the expected distance $\E \left[ d_W^{2}(\bmu_{p}, \nu_{0}) \right]$, from the standard one $\bigO\left( \frac{1}{p} \right)$ to slower rates. For example, by Theorem 5.1 in  \cite{W1}, the following upper bound holds
\begin{equation}
\E \left[ d_W^{2}(\bmu_{p}, \nu_{0}) \right] \leq \frac{2}{p+1} J_{2}(\nu_{0}), \label{eq:upBL}
\end{equation}
where the so-called $J_{2}$-functional is defined as
$$
J_{2}(\nu_{0}) = \int_{\Omega} \frac{F_{0}(x)(1-F_{0}(x))}{f_{0}(x)} dx,
$$
where $f_{0}$ is the density of $\nu_{0}$, and $F_{0}$ denotes its cdf. Therefore, provided that $J_{2}(\nu_{0})$ is finite, the empirical measure $\bmu_{p}$ converges to $\nu_{0}$ at the rate $\bigO\left( \frac{1}{p} \right)$.  Hence, using inequality \eqref{eq:upBL}, we have:

\begin{coro} \label{coro:ratehatbnu}
Suppose that Assumptions   \ref{A1}, \ref{A2} and \ref{A3} are satisfied. Then, the estimator  $\neb$ satisfies
\begin{equation}
\E \left[ d_W^{2}(\neb, \nu_{0}) \right] \leq \frac{1}{n}  \int_{0}^{1}  \var\left(  \bF^{-}  (\alpha) \right)d\alpha + \frac{2}{p+1} J_{2}(\nu_{0}). \label{eq:ratehatbnu}
\end{equation}
\end{coro}

By Corollary \ref{coro:ratehatbnu},  if  $J_{2}(\nu_{0}) < + \infty$, then it follows that $\neb$ converges to $\nu_{0}$ at the rate $\bigO\left( \frac{1}{n} + \frac{1}{p} \right)$. Hence, in the setting where $p \geq n$, $\neb$ converges at the classical parametric rate $\bigO\left( \frac{1}{n}  \right)$, provided that $J_{2}(\nu_{0}) < + \infty$. The case $p \geq n$ is usually refereed to as the dense case in the  literature on functional data analysis (see e.g.\ \cite{li2010} and references therein) which corresponds to the situation where the number of observations per unit/subject is larger than the sample size $n$ of functional objects. In the sparse case (when $p < n$), the non-smoothed Wasserstein barycenter converges at the rate $\bigO\left( \frac{1}{p}  \right)$, provided that $J_{2}(\nu_{0}) < + \infty$. 

%Using equality \eqref{eq:exactrate} in Theorem \ref{theo:ratehatbnu}, one may prove that the rate $\bigO\left( \frac{1}{n} + \frac{1}{p} \right)$ is sharp for some specific distributions $\nu_{0}$.

%Hence, equality \eqref{eq:rateunif} shows that the rate $\bigO\left( \frac{1}{n} + \frac{1}{p} \right)$ is sharp when the population Wasserstein barycenter is a distribution such that $J_{2}(\nu_{0}) < + \infty$.

\begin{rem}
When $\nu_{0}$ is the uniform distribution on $[0,1]$ one has that $J_{2}(\nu_{0}) < + \infty$, but we have shown that $\E \left[ d_W^{2}(\neb, \nu_{0}) \right]  \asymp \frac{1}{n} + \frac{1}{np} + \frac{1}{p^2}$. Hence, in this setting, $\neb$ converges at the parametric rate $\bigO\left( \frac{1}{n}  \right)$ provided that $p \geq \sqrt{n}$, which is a dense regime condition  weaker than $p \geq n$. 
\end{rem}

To conclude this discussion on the rate of convergence of the non-smoothed Wasserstein barycenter in the case of samples of equal size, we study in more detail the control of the rate of convergence of the term $\E \left[ d_W^{2}(\bmu_{p}, \nu_{0}) \right]$ in inequality \eqref{eq:upperboundrate}. As pointed out in many works (see for example \cite{BGU05,W1} and the references therein) the fact that the functional  $J_{2}(\nu_{0})$  is finite or not is the key point to control the convergence of the empirical measure $\bmu_{p}$ to the population measure $\nu_{0}$ in the Wasserstein space. Some known facts concerning $J_2$ are the following.
\begin{enumerate}
\item If $J_2(\nu_0)<+\infty$ then $\nu_0$ is supported on an interval of $\R$ and its density is a.e.\ strictly positive on this interval.
\item If $\nu_0$ is compactly supported with a density bounded away from zero or with a log-concave density then  $J_2(\nu_0)<+\infty$.
\item If the density of $\nu_0$ is of the form $C_\alpha e^{-|x|^\alpha}$ then $J_2(\nu_0)$ is finite if and only if $\alpha>2$. In particular, $J_2(\nu_0)=+\infty$ for the Gaussian distribution.
\end{enumerate}

Some further comments can be made in the case where $\nu_{0}$ is a Gaussian distribution. In this setting, one has that $J_2(\nu_0)=+\infty$ and the rate of convergence of $\E \left[ d_W^{2}(\bmu_{p}, \nu_{0}) \right]$ to zero is slower than $\bigO\left( \frac{1}{p}  \right)$. Indeed, from Corollary 6.14 in  \cite{W1}, if $\nu_{0}$ is the standard Gaussian distribution, then there exist two constants $c_1,c_2 > 0$ such that
\begin{equation}
c_1 \frac{\log (\log(p)) }{p} \leq \E \left[ d_W^{2}(\bmu_{p}, \nu_{0}) \right] \leq c_2 \frac{\log (\log(p)) }{p}. \label{eq:rateGauss}
\end{equation}
Hence, using again inequality \eqref{eq:upperboundrate} combined with the above upper bound, we have:
\begin{coro} \label{coro:ratehatbnuGauss}
Suppose that Assumptions   \ref{A1}, \ref{A2} and \ref{A3} are satisfied. If  $\nu_{0}$ is the standard Gaussian distribution, then the estimator  $\neb$ satisfies
\begin{equation}
\E \left[ d_W^{2}(\neb, \nu_{0}) \right] \leq \frac{1}{n}  \int_{0}^{1}  \var\left(  \bF^{-}  (\alpha) \right)d\alpha + c \frac{\log (\log(p)) }{p},
\end{equation}
for some numerical constant $c  > 0$.
\end{coro}

Hence by Corollary \ref{coro:ratehatbnuGauss}, if $p$ is sufficiently large with respect to $n$ (namely when $p \geq n  \log(\log(p))$), then $\neb$ also converges at the classical parametric rate $\bigO\left( \frac{1}{n}  \right)$ when $\nu_{0}$ is the standard Gaussian distribution.

\begin{rem}
Following the work of \cite{W1}, if $\nu_{0}$ is a log-concave distribution, then one may obtain rates of convergence for $\E \left[ d_W^{2}(\neb, \nu_{0}) \right]$ that are slower than the standard  $\bigO\left( \frac{1}{p}  \right)$ rate (e.g.\ for beta or exponential distributions).   Moreover, it is also possible to considerer for any $q \geq 1$ and for any  probability measure  $\mu$   on the real line (with density $f$ and distribution function $F$) the functional
$$
J_q(\mu)=\int_\R\frac{\left(F(x)(1-F(x))\right)^{q/2}}{f(x)^{q-1}} dx
$$
in order to control the rate of convergence of the empirical measure to $\mu$ for the $q$-Wasserstein distance. 
\end{rem}

\subsection{Non-smoothed empirical barycenter in the general case}

Let us now consider the general situation where the $p_i$'s are possibly different. The result below gives an upper bound on the rate of convergence of $\nebpbar$ where $\ubar{p}=\left(p_1,\ldots,p_n\right)$. 

\begin{theo} \label{theo:ratehatbnupdiff}
Suppose that  Assumptions   \ref{A1}, \ref{A2} and \ref{A3} are satisfied. Then,
$$
\E \left[ d_W(\nebpbar, \nu_{0}) \right] \leq n^{-1/2} \sqrt{  \int_{0}^{1} \var\left(  \bF^{-}  (\alpha) \right)d\alpha } +  \frac{1}{n} \sum_{i=1}^{n} \sqrt{ \E \left[ d^2_W(\tilde{\bnu}_{i}, \bnu_{i}) \right] },
$$
where $\tilde{\bnu}_{i} = \frac{1}{p_{i}} \sum_{j=1}^{p_{i}} \delta_{X_{i,j}}$ for each $1 \leq i \leq n$
\end{theo}
For the random measure $\bnu$, we define the random variable 
$$
J_{2}(\bnu) = \int_{\Omega} \frac{\bF(x)(1-\bF(x))}{\bfun(x)} dx.
$$
Since the $\bnu_{i}$'s are independent copies of $\bnu$ by applying inequality \eqref{eq:upBL}, it follows that
$$
\sqrt{ \E \left[ d^2_W(\tilde{\bnu}_{i}, \bnu_{i}) \right] } \leq \sqrt{2     \E \left[ J_{2}(\bnu) \right] } p_{i}^{-1/2}.
$$
Hence,  from Theorem \ref{theo:ratehatbnupdiff}, we finally obtain the following upper bound on the rate of convergence for the non-smoothed empirical barycenter

\begin{coro} \label{coro:ratehatbnupdiff}
Suppose that  Assumptions   \ref{A1}, \ref{A2} and \ref{A3} are satisfied. %, and assume that $\Omega$ is a compact interval.
 If $J_{2}(\bnu)$ has a finite expectation, then
$$
\E \left[ d_W(\nebpbar, \nu_{0}) \right] \leq n^{-1/2} \sqrt{  \int_{0}^{1} \var\left(  \bF^{-}  (\alpha) \right)d\alpha } + \sqrt{2  \E \left[ J_{2}(\bnu) \right] } \left( \frac{1}{n} \sum_{i=1}^{n} p_{i}^{-1/2} \right).
$$
\end{coro}

From Corollary \ref{coro:ratehatbnupdiff}, one has that if $\min_{1 \leq i \leq n} p_{i} \geq n$ (dense case), then $ \frac{1}{n} \sum_{i=1}^{n} p_{i}^{-1/2} \leq n^{-1/2}$, and thus, the non-smoothed empirical barycenter converges of the parametric rate $n^{-1/2}$ (provided that $\E\left[J_{2}(\bnu) \right]<+\infty$), namely
\begin{equation}
\E \left[ d_W(\nebpbar, \nu_{0}) \right] \leq  \left( \sqrt{  \int_{0}^{1} \var\left(  \bF^{-}  (\alpha) \right)d\alpha } + \sqrt{2  \E \left[ J_{2}(\bnu) \right) } \right) n^{-1/2}. \label{eq:nonsmoothparamrate}
\end{equation}

\begin{rem}
Knowing if  $J_{2}(\bnu)$ has a finite expectation is in general a difficult task. But, if we assume that the density $\bfun$ of $\bnu$ is  bounded below by a non-random positive constant then (obviously) $\E\left[J_{2}(\bnu) \right]<+\infty$.
\end{rem}

%\begin{rem}
%In the proof Theorem \ref{theo:ratehatbnupdiff}, the compactness assumption on $\Omega$ is used to ensure that Fubini's theorem is valid to obtain the equality
%$$
% \E  \left(  \int_{0}^{1}\left[ \bar{\bF}_{n}^{-}(\alpha) - F_{0}^{-}(\alpha) \right]^2 d\alpha \right) =  \int_{0}^{1} \E \left[ \bar{\bF}_{n}^{-}(\alpha) - F_{0}^{-}(\alpha) \right]^2 d\alpha.
%$$
%This assumption can be relaxed provided the above equality remains valid.
%\end{rem}
%
\subsection{The case of smoothed empirical barycenters}

In this section, we assume that $\Omega=[0,1]$ and we discuss the rate of convergence of smoothed empirical barycenters  $\seb$ (note that the following results hold if $\Omega$ is any compact interval). 

To choose an appropriate kernel function to study the convergence rate of the estimator $\seb$, we follow the proposal made in \cite{Pana15}. We let $\psi$ be a positive, smooth and symmetric density on the real line, such that  $\int_{\R} x^2 \psi(x) dx = 1$. We also denote by $\Psi$ the  cdf of the density $\psi$ and, for a bandwidth parameter $h > 0$, we let $\psi_{h}(x) = \frac{1}{h} \psi \left( \frac{x}{h} \right)$. Then, for any  $y \in [0,1]$ and $h > 0$, we denote by $\mu_{h}^{y}$ the measure supported on $[0,1]$ whose density $f_{\mu_{h}^{y}}$ is defined as 
\begin{equation} \label{eq:kernchoice}
f_{\mu_{h}^{y}}(x) = \psi_{h}(x-y) + 2 b_2  \psi_{h}(x-y) \1_{\{x-y  > 0\}} + 2 b_1 \psi_{h}(x-y) \1_{\{x-y  < 0\}} + 4b_1 b_2, \quad x \in [0,1],
\end{equation}
where $b_1 = 1 - \Psi \left( (1-y)/h \right)$ and $b_2 = \Psi \left( - y / h \right)$. Then, for each $1 \leq i \leq n$, we construct a  kernel density estimator of $\bfun_{i}$ by defining $\hat{\bfun}^{h_{i}}_{i}$ as the density associated to the measure
\begin{equation}
\hat{\bnu}^{h_{i}}_{i} =  \frac{1}{p_{i}} \sum_{j=1}^{p_{i}}  \mu_{h_{i}}^{X_{i,j}},  \label{eq:kernsmooth}
\end{equation}
where $h_{i} > 0$ is a bandwidth parameter depending on $i$. For a discussion on the intuition for this choice of kernel smoothing, we refer to \cite{Pana15}. A key property to analyze the convergence rate of $\seb$ is the following lemma which relates the Wasserstein distance between $\hat{\bnu}^{h_{i}}_{i}$ and the empirical measure $\tilde{\bnu}_{i} = \frac{1}{p_{i}} \sum_{j=1}^{p_{i}} \delta_{X_{i,j}}$. 

\begin{lem} \label{lemma:kernsmooth}
Let $1 \leq i \leq n$. Suppose that $0 < h_{i} \leq 1/4$, then one has the following upper bound 
\begin{equation}
d_W^{2}(\hat{\bnu}^{h_{i}}_{i}, \tilde{\bnu}_{i} ) \leq 3 h_i^2 + 4  \Psi ( - 1 / \sqrt{h_i} ),\qquad  1\leq i\leq n.\label{eq:boundkernsmooth}
\end{equation}
% Furthermore, if  the density $\psi$, used to define kernel smoothing in \eqref{eq:kernsmooth}, is such that there exists a positive constant $C$ and $\alpha \geq 5$ satisfying
 Furthermore, if there exist  constants $C > 0$ and $\alpha \geq 5$ satisfying
\begin{equation}
\psi(x) \leq C x^{- \alpha}, \mbox{ for all sufficiently large } x,  \label{eq:asspsi}
\end{equation}
then for $h_i$ small enough 
$$
d_W^{2}(\hat{\bnu}^{h_{i}}_{i}, \tilde{\bnu}_{i} ) \leq C_{\psi} h_i^2,
$$
for some constant $C_{\psi} > 0$ depending only on $\psi$.
\end{lem}
\begin{proof}
The upper bound \eqref{eq:boundkernsmooth} follows immediately from Lemma 1 in \cite{Pana15} and the symmetry of $\psi$. Then, by applying inequality \eqref{eq:asspsi} and since $\psi$ is symmetric, it follows that for $h$ small enough
$$
\Psi (- 1 / \sqrt{h} ) = \int_{-\infty}^{- 1 / \sqrt{h} } \psi(x) dx =  \int_{1 / \sqrt{h}}^{+ \infty } \psi(x) dx \leq C \int_{1 / \sqrt{h}}^{+ \infty } x^{- \alpha} dx = \frac{C}{\alpha - 1} h^{(\alpha -1)/2}.
$$
Hence, the second part of Lemma \ref{lemma:kernsmooth} is a consequence of the above inequality, the fact that $\alpha \geq 5$, and the upper bound \eqref{eq:boundkernsmooth}, which completes the proof.
\end{proof}

The result below gives a rate of convergence for the estimator $\seb$. 

\begin{theo} \label{theo:ratehatbnuh}
Suppose that  Assumptions   \ref{A1}, \ref{A2} and \ref{A3} are satisfied, and  that the density $\psi$, used to define kernel smoothing in \eqref{eq:kernsmooth}, satisfies inequality \eqref{eq:asspsi}. If $J_{2}(\bnu)$ has a finite expectation, and the bandwidth parameters $h_i$ are
%such that $\max\{h_i,1\leq i\leq n\}$ is
small enough, then we have
\begin{equation}
\E \left[ d_W(\seb, \nu_{0}) \right] \leq n^{-1/2} \sqrt{  \int_{0}^{1} \var\left(  \bF^{-}  (\alpha) \right)d\alpha } + C_{\psi}^{1/2} \left( \frac{1}{n} \sum_{i=1}^{n} h_{i} \right) + \sqrt{2  \E \left[ J_{2}(\bnu) \right] } \left( \frac{1}{n} \sum_{i=1}^{n} p_{i}^{-1/2} \right). \label{eq:boundsmooth}
\end{equation}
\end{theo}

Theorem \ref{theo:ratehatbnuh} can then be used to discuss choices of bandwidth parameters that may lead to a parametric rate of convergence. For example, if $0 < h_{i} \leq n^{-1/2}$ for all $1 \leq i \leq n$ and $\min_{1 \leq i \leq n} p_{i} \geq n$ (dense case), then Theorem \ref{theo:ratehatbnuh} implies that (for all sufficiently large $n$ to ensure that $\max_{1\leq i\leq n}\{h_i\}$ is small enough)
\begin{equation}
\E \left[ d_W(\seb, \nu_{0}) \right] \leq  \left( \sqrt{  \int_{0}^{1} \var\left(  \bF^{-}  (\alpha) \right)d\alpha } + C_{\psi}^{1/2}   + \sqrt{2  \E \left[ J_{2}(\bnu) \right) } \right) n^{-1/2}. \label{eq:paramrate}
\end{equation}

\begin{rem}
In the dense case (namely $\min_{1 \leq i \leq n} p_{i} \geq n$), by comparing the upper bounds \eqref{eq:nonsmoothparamrate} and \eqref{eq:paramrate}, it can be seen that  a preliminary smoothing step of the data (namely kernel smoothing the empirical measures $\tilde{\bnu}_{i} = \frac{1}{p_{i}} \sum_{j=1}^{p_{i}} \delta_{X_{i,j}}$) does not improve the parametric rate of convergence $n^{-1/2}$. Moreover, the bandwidths values has to be small  to ensure  the rate of convergence $n^{-1/2}$ for $\E \left[ d_W(\seb, \nu_{0}) \right]$. This result comes from the fact that  we evaluate the risk of empirical barycenters at the level of measures in $\WS$, and that we do not aim to control an estimation of the density $f_{0}$ of the population mean measure $\nu_{0}$.
\end{rem}

\begin{rem}
Theorem \ref{theo:ratehatbnuh} shares similarities with the results from Theorem 2 in \cite{Pana15} which gives the rate of convergence for smoothed Wasserstein barycenters computed from the realizations of multiple Poisson processes in a deformable model of measures similar to that of this paper. The main difference in  \cite{Pana15} is that the number $\bp_{i}$ of observations for each experimental unit are independent Poisson random variables with expectation $\E(\bp_{i}) = \tau_{n}$ for each $1 \leq i \leq n$ (they are not deterministic integers). From such observations and under similar assumptions, it is proved in \cite{Pana15} that the following upper bound holds (in probability)
\begin{equation}
d_W(\seb, \nu_{0}) \leq \bigO_{\P}\left( \frac{1}{\sqrt{n}} \right) +  \bigO_{\P}\left(  \frac{1}{n} \sum_{i=1}^{n} h_{i}  \right) +  \bigO_{\P}\left(  \frac{1}{\sqrt[4]{\tau_{n}}}  \right). \label{eq:boundPana15}
\end{equation}
Hence, under the conditions that $\tau_{n} \geq \bigO(n^{2})$ and $\max_{1 \leq i \leq n} h_{i} \leq \bigO_{\P}\left( n^{-1/2} \right)$, it follows from Theorem 2 in \cite{Pana15} that $\seb$ converges at the parametric rate $\bigO\left(  n^{-1/2} \right)$, for the Wasserstein distance. The quantity $\tau_{n}$ represents the averaged number of points observed for each Poisson process. As remarked in  \cite{Pana15} the condition $\tau_{n} \geq \bigO(n^{2})$ corresponds to a dense sampling regime where the number $n$ of observed Poisson processes should not grow too fast with respect to the expected  number of points observed for each process. Comparing the upper bounds \eqref{eq:boundsmooth} and \eqref{eq:boundPana15}, the main difference in the control of the risk of $\seb$  between our approach and the one in \cite{Pana15} is that we use the condition $\E \left[ J_{2}(\bnu) \right]  < + \infty$. Under such an assumption, the smoothed Wasserstein barycenter (for the model considered in this paper) may be shown to converge at the rate $\bigO\left(  n^{-1/2} \right)$, for the expected Wasserstein distance, under the dense case  setting $p: = \min\{ p_{i},\ 1 \leq i \leq n\} \geq n$ which is somehow a weaker condition than $\E(\bp_{i}) \geq n^{2}$ for all $1 \leq i \leq n$ as in \cite{Pana15}.
\end{rem}

% Remark 1. The assumption that ?n/ log n ? ? requires that the number of observed processes should not grow too rapidly relative to the mean number of points observed per process. This condition can be compared to similar conditions relating the number of discrete observations per curve in classical FDA. In a sense, it separates the so-called sparse from the dense sampling regime 

% is the number of points observed for the ith process.

\subsection{A lower bound on the minimax risk}

In the rest of this section, we show that, in the dense case and for the expected squared Wasserstein distance, the rate of convergence $\bigO\left(  n^{-1} \right)$ for non-smoothed empirical Wasserstein barycenters is optimal from the minimax point of view over a large class of random measures $\bnu$ satisfying the deformable model defined in Section \ref{sec:deform} through Assumptions  \ref{A1}, \ref{A2} and \ref{A3}.

\begin{defin} \label{defin:classD}
For  $\nu_{0} \in \WSac$ and  $\sigma > 0$, we define $\DD(\Omega,\nu_{0}, \sigma^2)$ as the class of $\WS$-valued random measures $\bnu$ that satisfy the deformable model defined in Section \ref{sec:deform} with
$
\var(\bnu) <\sigma^2.
$
%When $\Omega=[0,1]$ we write $\DD^0=\DD([0,1],\nu_{0}, \sigma^2)$.
\end{defin}

\begin{defin} \label{defin:classF}
Let $A > 0$. We denote by $\FF(\R,A) \subseteq W_2^{ac}(\R)$ a given set of measures with variance bounded by $A$, which contains at least all Gaussian distributions with variance bounded by $A$.
\end{defin}

Then, by  inequality \eqref{eq:upperboundrate}, we obtain the following corollary giving a uniform  rate of convergence for the non-smoothed empirical barycenter in the case of samples of equal size.

\begin{coro} \label{eq:corroupperboundpequal}
Let $A > 0$ and $\sigma > 0$. Suppose that $p_1 = p_2 = \ldots = p_n = p$. Then, if there exists a constant $c_{0} > 0$ such that
\begin{equation}
\sup_{\nu_{0} \in \FF(\R,A) }   \E \left[ d_W^{2}(\bmu_{p}, \nu_{0}) \right] \leq \frac{c_{0}}{n}, \label{eq:conddense}
\end{equation}
it follows that
\begin{equation}
\sup_{\nu_{0} \in \FF(\R,A) } \sup_{ \bnu \in \DD(\R,\nu_{0}, \sigma^2) } \E \left[ d_W^{2}(\neb, \nu_{0}) \right] \leq \frac{\sigma^2 + c_{0}}{n}. \label{eq:optrate}
\end{equation}
\end{coro}

The condition \eqref{eq:conddense} may be interpreted as the generalization of the dense case setting that has been discussed in the previous sections as it is valid only if $p$ is sufficiently large with respect to $n$. As an example, let $A \geq 0$ and suppose that the set $\FF(\R,A)$ can be partitioned as
$$
\FF(\R,A) = \FF_{0}(\R,A) \cup \GG(\R,A),
$$
where $\FF_{0}(\R,A)$ denotes a set of measures $\nu_{0} \in W_2^{ac}(\R)$ with variance bounded by $A$  satisfying
$$
A_{0} := \sup_{\nu_{0} \in \FF_{0}(\R,A)}   J_{2}(\nu_{0})  < + \infty,
$$
while $\GG(\R,A)$ denotes the set of Gaussian distributions with variance bounded by $A$. For this example, it follows from inequalities \eqref{eq:upBL} and \eqref{eq:rateGauss} in Section \ref{sec:convratepequal} (with samples of equal size) that 
$$
\sup_{\nu_{0} \in \FF(\R,A) }  \E \left[ d_W^{2}(\bmu_{p}, \nu_{0}) \right] \leq  \max \left(  A_{0} \frac{2}{p+1} , c_2 A  \frac{\log (\log(p)) }{p} \right) \leq \max(2 A_{0} ,c_2 A)   \frac{\log (\log(p)) }{p},
$$
provided that $\log(\log(p)) \geq 1$, where $c_{2}$ is a constant from inequality \eqref{eq:rateGauss}. Hence, if $p$ is such that $p \geq n  \log (\log(p))$ then condition \eqref{eq:conddense} is satisfied with
$$
c_{0} = \max(2 A_{0} ,c_2 A) = \max \left( 2 \sup_{\nu_{0} \in \FF_{0}(\R,A)}   J_{2}(\nu_{0}), c_2 A \right) .
$$

The following theorem shows that the upper bound \eqref{eq:optrate}  in Corollary \ref{eq:corroupperboundpequal}   is optimal (in term of rate of convergence) from the minimax point of view in nonparametric statistics. 

\begin{theo} \label{theo:lowerbound}
Let $A > 0$ and $\sigma > 0$. Then, the following lower bound holds
\begin{equation}
\inf_{\hat{\bnu}} \sup_{\nu_{0} \in \FF(\R,A) } \sup_{ \bnu \in \DD(\R,\nu_{0}, \sigma^2) } \E \left[ d_W(\hat{\bnu}, \nu_{0}) \right] \geq \frac{e^{-2} \min(A^{1/2},\sigma) }{4}n^{-1/2},
\end{equation}
where $\hat{\bnu} = \hat{\nu}\left( (X_{i,j})_{1 \leq i \leq n; \; 1 \leq j \leq p_{i} } \right) $ denotes any estimator taking values in $(W_2(\R),\BB \left( W_2(\R) \right))$ with $\hat{\nu}$ denoting a measurable function of the data  $(X_{i,j})_{1 \leq i \leq n; \; 1 \leq j \leq p_{i} }$ sampled from  the deformable model defined in Section \ref{sec:deform}. 
\end{theo}

% in the upper bounds \eqref{eq:nonsmoothparamrate} and \eqref{eq:paramrate} is optimal 
Now, by using inequalities \eqref{eq:nonsmoothparamrate} and \eqref{eq:paramrate}  and Definitions \ref{defin:classD} and \ref{defin:classF} introduced above, we  also obtain the following corollary giving uniform rates of convergence for the non-smooth Wasserstein barycenter in the general situation where the $p_i'$s are possibly different.

\begin{coro} \label{eq:corroupperbound}
 Let $A > 0$ and $\sigma > 0$. Suppose that the assumptions of Corollary \ref{coro:ratehatbnupdiff} are satisfied, and   that $ p_{i} \geq n$,  for all $1 \leq i \leq n$. Then, the following upper bound holds 
$$
\sup_{\nu_{0} \in \FF(\R,A) } \sup_{ \bnu \in \DD(\R,\nu_{0}, \sigma^2) } \E \left[ d_W(\nebpbar, \nu_{0}) \right] \leq  n^{-1/2}  \left( \sigma  + \sqrt{2}  \sup_{\nu_{0} \in \FF(\R,A) } \sup_{ \bnu \in \DD(\R,\nu_{0}, \sigma^2) }  \sqrt{ \E \left[ J_{2}(\bnu) \right] }  \right) .
$$
\end{coro}

Hence, under the assumptions made in Corollary \ref{eq:corroupperbound}, the estimator $\nebpbar$ converges at the optimal rate of convergence $n^{-1/2}$ provided that 
$$
\sup_{\nu_{0} \in \FF(\R,A) } \sup_{ \bnu \in \DD(\R,\nu_{0}, \sigma^2) }  \E \left[ J_{2}(\bnu) \right] < + \infty.
$$ 
We conclude this discussion by a few remarks on the rate of convergence that may be obtained in the sparse case.

\begin{rem}
In the case of samples of equal size, the results above show that the rate of convergence $n^{-1}$ is optimal in the dense case (for the risk $\E \left[ d_W^2(\neb, \nu_{0}) \right]$), namely when the number $p = p_1 = \ldots = p_n$ of observations per units is sufficiently large with respect to $n$. We believe that deriving a lower bound on the minimax risk depending on $p$ in the sparse case (e.g.\ when $p < n$) is more involved. Indeed, from the discussion in Section \ref{sec:convratepequal} on the rate of convergence non-smoothed empirical barycenter, it appears that the exact decay of $\E \left[ d^{2}_W(\neb, \nu_{0}) \right]$ as a function of $p$ is difficult to establish as it depends on $\nu_{0}$. Indeed, from Section \ref{sec:convratepequal}, one has that
\begin{itemize}
\item[-] if $\nu_{0}$ is the uniform distribution on $[0,1]$, then $\E \left[ d_W^{2}(\neb, \nu_{0}) \right]  \asymp \frac{1}{n} + \frac{1}{np}+ \frac{1}{p^2} $,
\item[-] if $\nu_{0}$ is the one-sided exponential distribution, then $\E \left[ d_W^{2}(\neb, \nu_{0}) \right] =  \bigO\left( \frac{1}{n} +  \log(p) \left( \frac{1}{np} + \frac{1}{p} \right) \right)$, 
\item[-] if $\nu_{0}$ is the standard Gaussian distribution, then $\E \left[ d_W^{2}(\neb, \nu_{0}) \right] =  \bigO\left( \frac{1}{n} +  \log(\log(p)) \left( \frac{1}{np} + \frac{1}{p} \right) \right)$, 
\item[-] if $\nu_{0}$ is such that $J_{2}(\nu_{0}) < + \infty$, then $\E \left[ d_W^{2}(\neb, \nu_{0}) \right] =  \bigO\left( \frac{1}{n} +  \frac{1}{p} \right)$. 
\end{itemize}
From Theorem \ref{theo:ratehatbnu}, one has that the risk of the non-smoothed empirical barycenter may be bounded from below as follows
\begin{equation}
\sup_{\nu_{0} \in \FF(\R,A) } \sup_{ \bnu \in \DD(\R,\nu_{0}, \sigma^2)}  \E \left[ d_W^{2}(\neb, \nu_{0}) \right] \geq  \sup_{\nu_{0} \in \FF(\R,A) } \sum_{j=1}^{p} \int_{(j-1)/p}^{j/p}   \left( \E \left[ Y_{j}^{\ast}  \right]- F_{0}^{-}(\alpha) \right)^2 d\alpha. \label{eq:lowerboundp}
\end{equation}
The quantity $\sum_{j=1}^{p} \int_{(j-1)/p}^{j/p}   \left( \E \left[ Y_{j}^{\ast}  \right]- F_{0}^{-}(\alpha) \right)^2 d\alpha$ may be interpreted as a bias term when estimating the unknown measure by the nonparametric estimator $\bmu_{p} = \frac{1}{p}\sum_{j=1}^{p} \delta_{Y_{j}}$. Therefore, for samples of equal size and in the sparse case (when $p < n$), the lower bound \eqref{eq:lowerboundp} may be used to control (as a function of $p$) the best rate of convergence for $\neb$ that may be obtained over the class of measures $\nu_{0} \in \FF(\R,A)$.
\end{rem}

\begin{rem}
Finally, we remark that better rates of convergence may be obtained if one assumes a parametric model for the random measure $\bnu$. Indeed, suppose that $\mu_{0} \in \WSac$ denotes a {\it known probability measure} with expectation $m_{0}$ and variance $\sigma_{0}^2$ %$m_{0} = \int_{\Omega} x d\mu_{0}(x)$,
% such that its quantile function $F_{\mu_{0}}^{-}$ is $L$-Lipschitz on $[0,1]$,
and consider that the data $(X_{i,j})_{1 \leq i \leq n; \; 1 \leq j \leq p_{i} }$ are sampled from iid random measures $\bnu_{1},\ldots,\bnu_{n}$ satisfying the location model
\begin{equation}
F_{\bnu_{i}}^{-}(\alpha) = F_{\mu_{0}}^{-}(\alpha) + \ba_{i}, \; \alpha \in [0,1], \; 1 \leq i \leq n, \label{eq:locmodel}
\end{equation}
where $\ba_{1},\dots,\ba_{n}$ are iid random variables with unknown expectation $\bar{a}$ and variance $\gamma^{2}$. In this model, the population Wasserstein barycenter is the measure $\nu_{0}$ with quantile function $F_{\nu_{0}}^{-}(\cdot)  = F_{\mu_{0}}^{-}(\cdot) + \bar{a}$. Since, the measure $\mu_{0}$ is assumed to be known, a natural estimator for $\nu_{0}$ is to take the measure $\hat{\bnu}_{0}$ with quantile function   $F_{\hat{\bnu}_{0}}^{-}(\cdot)  = F_{\mu_{0}}^{-}(\cdot) + \hat{\ba}$, with
$$
\hat{\ba} = \frac{1}{n} \sum_{i=1}^{n} \frac{1}{p_{i}} \sum_{j=1}^{p_{i}} X_{ij} - m_{0}.
$$
Then, it is clear that
\begin{eqnarray*}
\E \left[ d_W^{2}(\hat{\bnu}_{0}, \nu_{0}) \right] & = & \int_{0}^{1} \E  \left(F_{\hat{\bnu}_{0}}^{-}(\alpha) - F_{\nu_{0}}^{-}(\alpha)  \right)^2 d\alpha = \E  \left( \hat{\ba} -  \bar{a} \right)^2  \\
%& = &  \frac{1}{n^{2}} \sum_{i=1}^{n} \frac{1}{p_{i}^{2}} \var \left( \sum_{j=1}^{p_{i}} X_{ij} \right) \\
%& = & \frac{1}{n^{2}} \sum_{i=1}^{n} \frac{1}{p_{i}^{2}}  \left(p_{i} (\sigma_{0}^2 + \gamma^{2}) + p_{i}(p_{i} -1) \gamma^{2}  \right)   \\
& = & \frac{\sigma_{0}^2 + \gamma^{2}}{n} \left(   \frac{1}{n}   \sum_{i=1}^{n} \frac{1}{p_{i}}  \right) +  \frac{\gamma^{2}}{n} \left(   \frac{1}{n}   \sum_{i=1}^{n} \frac{p_{i}-1}{p_{i}}  \right).
\end{eqnarray*}
In the case where all the $p_i$'s are equal to $p$, then the above equality simplifies to
$$
\E \left[ d_W^{2}(\hat{\bnu}_{0}, \nu_{0}) \right]  = \frac{\sigma_{0}^2 + \gamma^{2}}{n p}+  \frac{\gamma^{2}}{n}  \frac{p-1}{p},
$$
and thus the parametric estimator $\hat{\bnu}_{0}$ converges at the rate $\bigO\left(\frac{1}{n} + \frac{1}{np} \right)$. Therefore, either in the dense ($p \geq n$) or sparse case  ($p < n$), the parametric estimator $\hat{\bnu}_{0}$ converges at the rate $\bigO\left(\frac{1}{n} \right)$ in the location model \eqref{eq:locmodel} when the ``reference measure'' $\mu_{0}$ is known. Moreover, in the sparse case ($p < n$), the parametric estimator $\hat{\bnu}_{0}$ converges faster than the non-smoothed empirical Wasserstein barycenter $\neb$ thanks to the results in Section \ref{sec:convratepequal}. 
\end{rem}

\section{Numerical experiments} \label{sec:num}

In this simulation study, we perform Monte Carlo experiments to compare the decay of the squared Wassertein risks $\E \left[ d_W^{2}(\seb, \nu_{0}) \right]$ and $\E \left[ d_W^{2}(\neb, \nu_{0}) \right]$ of the smoothed and non-smoothed empirical Wasserstein barycenters $\seb$ and $\neb$ as a function of the number $n$ of units and the sample size $p$. 

We analyze the case of random samples $(X_{i,j})_{1 \leq i \leq n; \; 1 \leq j \leq p }$ with $10 \leq n \leq 200$ and  $10 \leq p \leq 200$. Data are generated from densities supported on $\Omega = [-7,7]$ that are sampled from the following model accounting for vertical and horizontal variations
\begin{equation} \label{eq:locscalesim}
\bfun_{i}(x) =C_{i} \ba_i^{-1}f \left( \ba_i^{-1}(x-\bb_{i})\right), \; x \in \Omega, \; 1 \leq i \leq n.
\end{equation}
where $f$ is the density of the standard Gaussian law  on $\R$, $\ba_{i} \sim \U([0.8,1.2])$, $\bb_{i} \sim \U([-2,2])$, and $C_{i}$ is a normalizing constant such that $\bfun_{i}$ integrates to one on $\Omega$. This setting  corresponds to the the simulation study conducted in \cite{PetersenMuller}.

For given values of $n$ and $p$, we evaluate the Wasserstein risk of $\seb$ by repeating $M=100$ times the following experiment. First, data are simulated from model \eqref{eq:locscalesim}. Then, for each $1 \leq i \leq n$, we  use kernel smoothing to compute the density $\hat{\bfun}^{h_i}_{i}$ and its associated measure  $\hat{\bnu}_{i}^{h_{i}}$. We slightly deviate from the analysis carried out in Section \ref{sec:convrate}, as we use a Gaussian kernel to smooth the data $(X_{i,j})_{1 \leq j \leq p }$ with bandwidth $h_i$ chosen by cross validation, instead of the specific kernel defined in \eqref{eq:kernchoice} that has been proposed for the convergence analysis of $\seb$. We found that this modification has no substantial effect on the finite sample performance of the procedure, and a similar choice has been made in the numerical experiments in \cite{Pana15}. In Figure \ref{fig:simux07:data}(a), we display  an example of  densities estimated from realizations of the model \eqref{eq:locscalesim} with $n = p = 100$. After computing the quantile function $F_{\seb}^{-}$ of the empirical smoothed Wasserstein barycenter $\seb$, we approximate $d_W^{2}(\seb, \nu_{0}) =  \int_{0}^1 (F_{\seb}^{-}(\alpha) - F_{\nu_{0}}^{-}(\alpha))^2 d\alpha$ by discretizing the integral over a fine grid of values for $\alpha \in ]0,1[$. This approximated value of  $d_W^{2}(\seb, \nu_{0})$ is then averaged over the $M=100$ repeated experiments to approximate $\E \left[ d_W^{2}(\seb, \nu_{0}) \right]$. 

Thanks to the explicit expression \eqref{eq:nonregestim} of the non-smoothed empirical Wasserstein barycenter $\neb$,
its  quantile function $F_{\neb}^{-}$  is straightforward to compute on a grid of values for $\alpha$, and the Wasserstein risk  $\E \left( d_W^{2}(\neb, \nu_{0}) \right)$ is then approximated in the same way by using Monte Carlo repetitions.

For values of $n$ and $p$ ranging from 10 to 200, we display  in Figure \ref{fig:simux07:data} (c) and   \ref{fig:simux07:data} (d) these approximations of $\E \left[ d_W^{2}(\seb, \nu_{0}) \right]$ and $\E \left[ d_W^{2}(\neb, \nu_{0}) \right]$ (in logarithmic scale). For both estimators, it appears that the Wasserstein risk is clearly a decreasing function of the number $n$ of units. To the contrary increasing $p$ does not lead to a significant decay of this risk. This suggest that $ \frac{1}{n}  \int_{0}^{1}  \var\left(  \bF^{-}  (\alpha) \right)d\alpha $ is the most significant term in the upper bound \eqref{eq:ratehatbnu} of the Wasserstein risk of $\neb$.

In Figure \ref{fig:simux07:data} (b), we also display the logarithm of the ratio $\E \left[ d_W^{2}(\neb, \nu_{0})\right]  / \E\left[ d_W^{2}(\seb, \nu_{0})  \right]$. For values of $p$ larger than 100, both estimators (smoothed and non-smoothed empirical Wasserstein barycenters) appear to have squared Wasserstein risks of approximately the same magnitude. This tends to confirm the results on convergence rates obtained in Section \ref{sec:convrate} in the dense case (when $p$ is sufficiently large with respect to $n$) which show that a preliminary smoothing is not necessary in this setting. However, for smaller values of $p$ (between 10 and 50), the smoothed empirical Wasserstein barycenter has a smaller Wasserstein risk. This suggests that introducing a smoothing step through kernel smoothing of the data in each experimental unit improves the quality of the estimation of $\nu_{0}$ when the sample size $p$ is small, which corresponds to the sparse case.

\begin{figure}[htbp]
\centering
\subfigure[An example of estimated densities.]{\includegraphics[width=7cm]{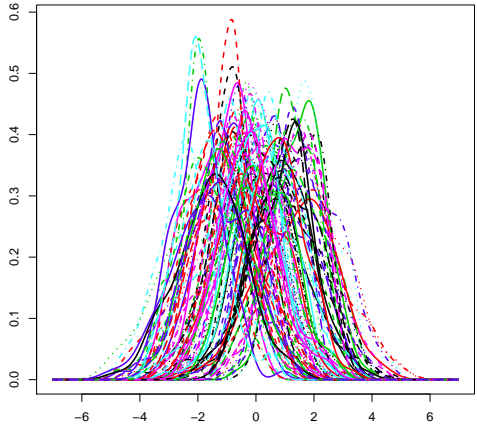}}
\subfigure[ $\log\left(\E \left( d_W^{2}(\neb, \nu_{0}) \right) / \E\left(d_W^{2}(\seb, \nu_{0})  \right)\right)$.]{\includegraphics[width=7cm]{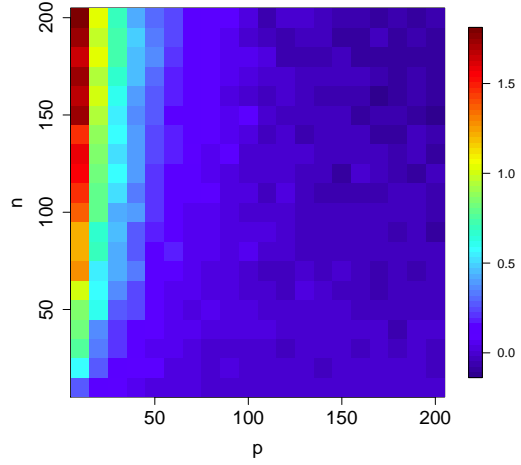}}

\subfigure[Log-Wasserstein risk of $\seb$.]{\includegraphics[width=7cm]{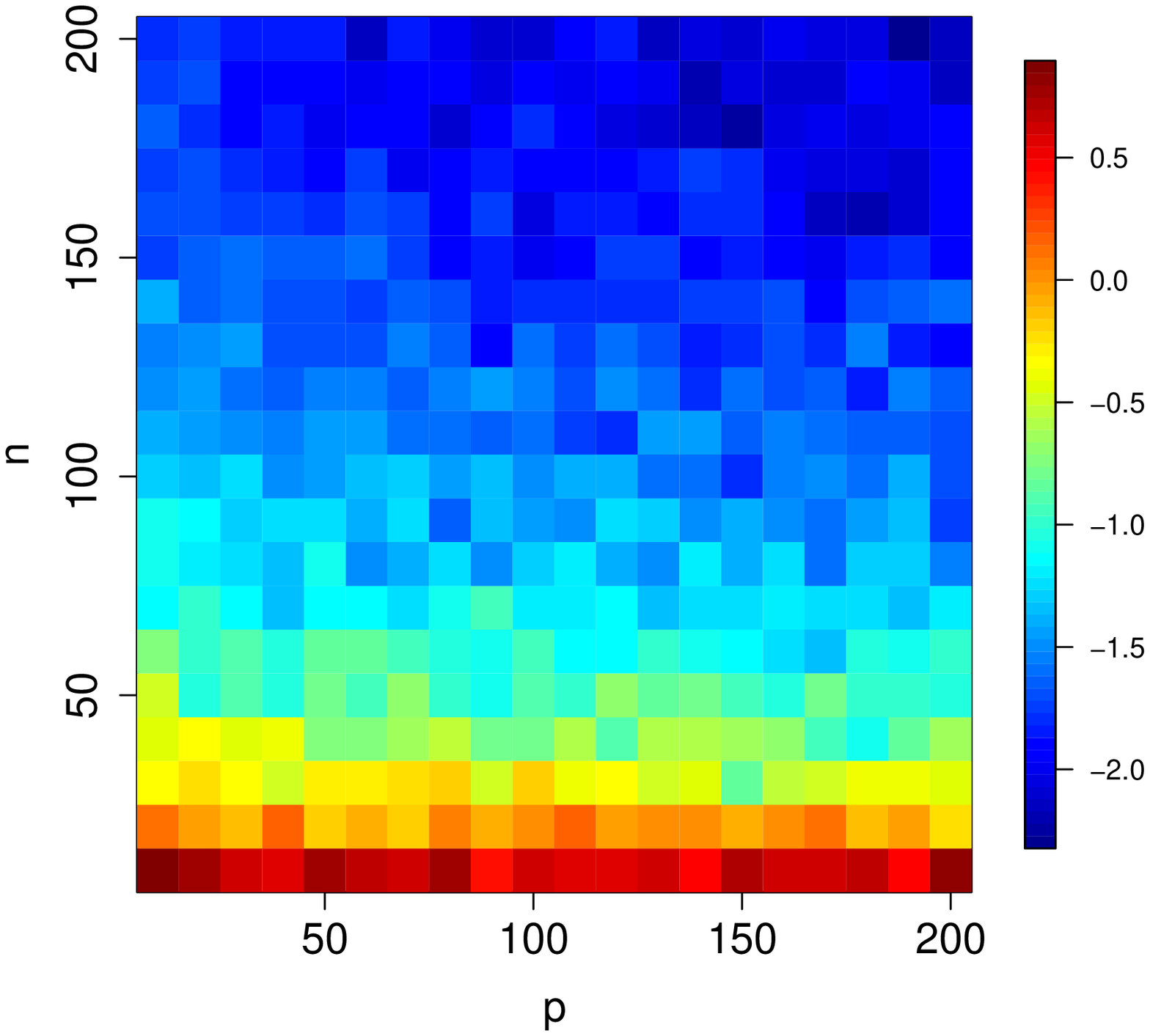}}
\subfigure[Log-Wasserstein risk of $\neb$.]{\includegraphics[width=7cm]{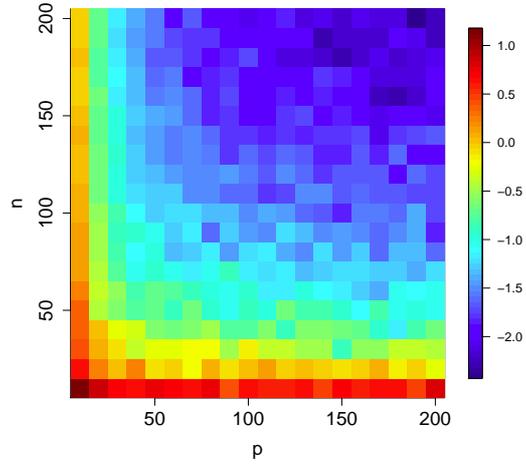}}

\caption{(a) An example of $n=100$  densities estimated from data sampled from model \eqref{eq:locscalesim} with $n = p = 100$, (b) Logarithm of the ratio  $\E ( d_W^{2}(\neb, \nu_{0}) ) / \E \left[ d_W^{2}(\seb, \nu_{0})  \right]$, (c) Wasserstein risk of the smoothed empirical barycenter $\seb$ with kernel bandwidths chosen by cross-validation,  (c) Wasserstein risk of the non-smoothed empirical barycenter $\neb$. The values of $n$ and $p$ vary from 10 to 200 by an increment of 10.} \label{fig:simux07:data}
\end{figure}

\section{Conclusion and perspectives} \label{sec:conclusion}

In this paper, we have studied the rate of convergence for the (squared) Wasserstein distance of (possibly smoothed) empirical barycenters in a deformable model of measures. The main contributions of this work can be summarized as follows. In the case of samples of equal size, we have derived a closed-form formula  for the risk of non-smooth empirical barycenter as a function of $n$ and $p$, which allows to derive sharp rates convergence whose rate of decay in $p$ depends on the population mean measure $\nu_{0}$. A second conclusion of the paper is that, in the dense case (when the minimal number $\min_{1 \leq i \leq n} p_{i} \geq n$ of  observations per unit is sufficiently large with respect to the number $n$ of observed measures), the non-smooth empirical barycenter converges at the parametric rate of convergence $n^{-1}$. Moreover, this rate is shown to be a lower bound on the decay of a novel notion of minimax risk in the deformable model of measures introduced in this paper.  In the dense case, the numerical experiments that have been carried out are in agreement with the theoretical results which show, that in this setting, one may only consider the non-smoothed empirical Wasserstein barycenter, and that a preliminary smoothing step is not necessary to obtain an optimal estimator.

A first perspective would be to find a lower bound on the minimax risk depending on $p$ in the sparse case. However, to this end, we believe that one has to first obtain sharper rates of convergence as a function of $p$ for the non-smooth empirical barycenter.

Finally, a natural perspective is to ask how these results can be extended to higher dimensional settings for measures supported on $\R^{d}$ with $d > 1$. However, we believe that this is far from being obvious as the results in this paper rely heavily on the closed form formula for Wasserstein barycenters in the one-dimensional setting though quantile averaging. Such results do not hold in higher-dimension for data sets consisting of iid random vectors sampled from unknown random measures supported on $\R^{2}$ or $\R^{3}$ for example.

%%%%%%%%%%%%%%%

\appendix
\section{Appendix}

\subsection{Auxiliary results}

We recall that $Y_1,\ldots,Y_p$ denote iid variables sampled from the measure $\nu_{0}$ (independently of the data), and that the associated empirical measure is $\bmu_{p} = \frac{1}{p}\sum_{j=1}^{p} \delta_{Y_{j}}$.  By Corollary 4.5 in \cite{W1}, it follows that
\begin{equation}
\E \left[ d_W^{2}(\bmu_{p}, \nu_{0}) \right]  = \frac{1}{p} \sum_{j=1}^{p} \var \left( Y_{j}^{\ast}  \right) + \sum_{j=1}^{p} \int_{(j-1)/p}^{j/p}   \left(\E\left[ Y_{j}^{\ast} \right] - F_{0}^{-}(\alpha) \right)^2 d\alpha. \label{eq:W2order}
\end{equation}
where $Y_1^{\ast} \leq Y_2^{\ast}  \leq \ldots \leq Y_p^{\ast}$ denote the order statistics of the sample $Y_1,\ldots,Y_p$. 

It is well known that the  $j$-th order statistic $Y_{j}^{\ast}$ admits the density (see e.g.\ \cite{W1})
\begin{equation}
f_{Y_{j}^{\ast}}(y) = \frac{p!}{(j-1)! (p-j)!}  f_{0}(y) [F_{0}(y)]^{j-1} [1-F_{0}(y)]^{p-j}  , \; y \in \Omega, \label{eq:densY}
\end{equation}
Moreover, under Assumption \ref{A2}, one has that, conditionally on $\bF_{i}$, the $j$-th order statistic $X_{i,j}^{\ast}$ admits the density
\begin{equation}  
f_{X_{i,j}^{\ast}}(x) = \frac{p!}{(j-1)! (p-j)!}  \bfun_{i}(x) [\bF_{i}(x)]^{j-1} [1-\bF_{i}(x)]^{p-j}  , \; x \in \Omega. \label{eq:densX}
\end{equation}

Let us recall the notation $\bar{X}_{j}^{\ast} = \frac{1}{n} \sum_{i=1}^{n} X_{i,j}^{\ast}$. Then, the following result holds.

\begin{lem} \label{lem:A}
If  Assumptions   \ref{A1}, \ref{A2} and \ref{A3} are satisfied, then, for each $1 \leq j \leq p$, one has that
$$
\E\left[ \bar{X}_{j}^{\ast} \right] = \E \left[ Y_{j}^{\ast}  \right].
$$
Moreover,
$$
 \frac{1}{p} \sum_{j=1}^{p}  \var \left( \bar{X}_{j}^{\ast}  \right) - \var \left( Y_{j}^{\ast}  \right) = \frac{1}{n}  \left(  \int_{0}^{1}  \var \left(  \bF^{-}(\alpha)   \right)  d\alpha   \right) +  \frac{1-n}{p n} \sum_{j=1}^{p}  \var \left( Y_{j}^{\ast}  \right).
 % \leq  \frac{1}{n}  \int_{0}^{1}  \var\left(  \bF^{-}  (\alpha) \right)d\alpha.
$$
\end{lem}

\begin{proof}
Let $1 \leq j \leq p$ and $1 \leq i \leq n$.  Thanks to the expression \eqref{eq:densX} for the density of $X_{i,j}^{\ast}$, one has that
\begin{eqnarray*}
\E\left[ X_{i,j}^{\ast}  | \bF_{i}  \right] & = &  \int_{\Omega} x \frac{p!}{(j-1)! (p-j)!}  \bfun_{i}(x) [\bF_{i}(x)]^{j-1} [1-\bF_{i}(x)]^{p-j} dx \\
 & = &  \int_{0}^{1} \bF_{i}^{-}(\alpha) \frac{p!}{(j-1)! (p-j)!}   \alpha^{j-1} (1-\alpha)^{p-j} d\alpha.
\end{eqnarray*}
where we used the change of variable $\alpha = \bF_{i}(x)$  to obtain the last equality.  By Proposition  \ref{prop:existence},   $\E\left[   \bF_{i}^{-} \right] = F_{0}^{-}(\alpha)$ for each $1 \leq i \leq n$. Therefore,  using Fubini's theorem, it follows that
\begin{eqnarray}
\E\left[ X_{i,j}^{\ast} \right] & = & \E\left[  \E\left[ X_{i,j}^{\ast} | \bF_{i} \right] \right]  \nonumber \\
& = &    \E\left[   \int_{0}^{1} \bF_{i}^{-}(\alpha) \frac{p!}{(j-1)! (p-j)!}   \alpha^{j-1} (1-\alpha)^{p-j} d\alpha \right] \nonumber \\
& = &  \int_{0}^{1}  \E\left[   \bF_{i}^{-}(\alpha) \right]  \frac{p!}{(j-1)! (p-j)!}   \alpha^{j-1} (1-\alpha)^{p-j} d\alpha \nonumber \\
& = &  \int_{0}^{1}   F_{0}^{-}(\alpha)    \frac{p!}{(j-1)! (p-j)!}   \alpha^{j-1} (1-\alpha)^{p-j} d\alpha \nonumber \\
& = & \int_{\Omega} y \frac{p!}{(j-1)! (p-j)!}  f_{0}(y) [F_{0}(y)]^{j-1} [1-F_{0}(y)]^{p-j}  dy = \E \left[ Y_{j}^{\ast}  \right], \label{eq:EXEY}
\end{eqnarray}
where we used the change of variable $y = F_{0}^{-}(\alpha) $ and the expression \eqref{eq:densY} for the density of $Y_{j}^{\ast}$ to obtain the last equality above. Given that $\E\left[ \bar{X}_{j}^{\ast}  \right] =  \frac{1}{n} \sum_{i=1}^{n} \E\left[   X_{i,j}^{\ast}  \right] $, the  first statement of Lemma \ref{lem:A} follows from equality \eqref{eq:EXEY}.

Now, let us prove the second statement of Lemma \ref{lem:A}. Thanks to the expression \eqref{eq:densX} for the density of $X_{i,j}^{\ast}$, one has that, for each $1 \leq j \leq p$ and  $1 \leq i \leq  n$
\begin{eqnarray}
\E \left[  |X_{i,j}^{\ast}|^2 \right] & =  & \E \left[ \E \left[  |X_{i,j}^{\ast}|^2 | \bF_{i} \right] \right] \nonumber \\
& = & \E \left[  \int_{\Omega} x^2 \frac{p!}{(j-1)! (p-j)!}  \bfun_{i}(x) [\bF_{i}(x)]^{j-1} [1-\bF_{i}(x)]^{p-j} dx \right] \nonumber  \\
 & =  &  \int_{0}^{1} \E\left[ \left| \bF_{i}^{-}(\alpha) \right|^2 \right] \frac{p!}{(j-1)! (p-j)!}   \alpha^{j-1} (1-\alpha)^{p-j} d\alpha, \label{eq:X2}
\end{eqnarray}
where we again use the change of  variable $\alpha = \bF_{i}(x)$, and Fubini's theorem to obtain the last equality. Similarly, by equality \eqref{eq:densY},  it follows that for each $1 \leq j \leq p$
\begin{eqnarray}
\E \left[  |Y_{j}^{\ast} |^2 \right] & = &  \int_{\Omega} y^2 \frac{p!}{(j-1)! (p-j)!}  f_{0}(y) [F_{0}(y)]^{j-1} [1-F_{0}(y)]^{p-j} dy  \nonumber  \\
 & =  &  \int_{0}^{1} \left| F_{0}^{-}(\alpha) \right|^2 \frac{p!}{(j-1)! (p-j)!}   \alpha^{j-1} (1-\alpha)^{p-j} d\alpha.  \label{eq:Y2}
\end{eqnarray}
Since  $\bar{X}_{j}^{\ast} = \frac{1}{n} \sum_{i=1}^{n} X_{i,j}^{\ast}$, we obtain by independence that
$$
 \var \left( \bar{X}_{j}^{\ast}  \right) = \frac{1}{n^2}   \sum_{i=1}^{n}  \var \left(  X_{i,j}^{\ast} \right) = \frac{1}{n^2}   \sum_{i=1}^{n}  \E \left[  |X_{i,j}^{\ast}|^2 \right] - \left| \E\left[ X_{ij}^{\ast} \right] \right|^2 . 
$$
Hence, using equalities \eqref{eq:EXEY}, \eqref{eq:X2} and \eqref{eq:Y2}, and the fact that $\E\left[ \left| \bF_{i}^{-}(\alpha) \right|^2 \right] = \E\left[ \left| \bF^{-}(\alpha) \right|^2 \right] $ for each $1 \leq i \leq n$, we obtain that
\begin{eqnarray*}
 \var \left( \bar{X}_{j}^{\ast}  \right) & = & \frac{1}{n}  \left(  \int_{0}^{1} \E\left[ \left| \bF^{-}(\alpha) \right|^2 \right] \frac{p!}{(j-1)! (p-j)!}   \alpha^{j-1} (1-\alpha)^{p-j} d\alpha - \left|\E \left[ Y_{j}^{\ast}  \right]  \right|^2  \right) \\
 & = & \frac{1}{n}  \left(  \int_{0}^{1} \E\left[ \left| \bF^{-}(\alpha) \right|^2 \right] \frac{p!}{(j-1)! (p-j)!}   \alpha^{j-1} (1-\alpha)^{p-j} d\alpha +  \var \left( Y_{j}^{\ast}  \right) - \E \left[  |Y_{j}^{\ast} |^2 \right]   \right) \\
& = & \frac{1}{n}  \left(  \int_{0}^{1} \left( \E\left[ \left| \bF^{-}(\alpha) \right|^2 \right] - \left| F_{0}^{-}(\alpha) \right|^2  \right) \frac{p!}{(j-1)! (p-j)!}   \alpha^{j-1} (1-\alpha)^{p-j} d\alpha +  \var \left( Y_{j}^{\ast}  \right)   \right) \\
  & = & \frac{1}{n}  \left(  \int_{0}^{1}  \var \left(  \bF^{-}(\alpha)   \right) \frac{p!}{(j-1)! (p-j)!}   \alpha^{j-1} (1-\alpha)^{p-j} d\alpha +  \var \left( Y_{j}^{\ast}  \right)  \right),
\end{eqnarray*}
where, to obtain the last inequalities, we used that $ \E\left[ \bF^{-} \right] = F_{0}^{-}$ by Proposition \ref{prop:existence}. 
Therefore, from the above equality, one finally obtains that
\begin{eqnarray*}
\frac{1}{p} \sum_{j=1}^{p}  \var \left( \bar{X}_{j}^{\ast}  \right) - \var \left( Y_{j}^{\ast}  \right) & = &   \frac{1}{n}  \left(  \int_{0}^{1}  \var \left(  \bF^{-}(\alpha)   \right)  d\alpha   \right) +  \frac{1-n}{p n} \sum_{j=1}^{p}  \var \left( Y_{j}^{\ast}  \right), % \\
% & \leq &    \frac{1}{n}  \left(  \int_{0}^{1}  \var \left(  \bF^{-}(\alpha)   \right)  d\alpha   \right),
\end{eqnarray*}
which completes the proof of Lemma \ref{lem:A}.
\end{proof}

%\subsection{Proof of Proposition \ref{prop:nu0}}
%
%For $\mu \in \WS$, we define the functional $H(\mu) = \E \left( d_W^{2}(\bnu,\mu) \right)$. Since $\bnu$ is assumed to be square-integrable in the sense of Assumption \ref{A1}, one has that $H(\mu) < + \infty$ for any $\mu$. By the definition of the Wasserstein distance and using Fubini's theorem, we obtain that
%$$
%H(\mu) = \E \left(  \int_{0}^{1} \left(F_{\bnu}^{-}(\alpha)  - F^{-}_{\mu}(\alpha) \right)^{2} d\alpha \right) = \int_{0}^{1} \E   \left(F_{\bnu}^{-}(\alpha)  - F^{-}_{\mu}(\alpha) \right)^{2} d\alpha.
%$$
%Then, let us recall that, by Assumption \ref{A3},  $\E \left( \bF^{-}(\alpha)   \right) = F_{0}^{-}(\alpha)$ for any $\alpha\in ]0,1[$. Hence, using the inequality  $\E \left( X-a\right)^{2} > \E \left( X - \E(X) \right)^{2}$ which holds for any squared integrable real random variable $X$ and real number $a \neq \E(X)$, we obtain that
%$$
%H(\mu) > \int_{0}^{1} \E   \left(F_{\bnu}^{-}(\alpha)  - F_{0}^{-}(\alpha) \right)^{2} d\alpha = \E \left(  \int_{0}^{1} \left(F_{\bnu}^{-}(\alpha)  -  F_{0}^{-}(\alpha) \right)^{2} d\alpha \right) = H(\nu_{0})
%$$
%for any $\mu \in \WS$ with $\mu \neq \nu_{0}$. Therefore, this shows that $\mu \mapsto H(\mu)$ has a unique minimum at $\mu = \nu_{0}$, which completes the proof of Proposition \ref{prop:nu0}.
 
\subsection{Proof of Theorem \ref{theo:ratehatbnu}}

By  Definition \ref{defi:wdist} of the Wasserstein distance, and since $\neb = \frac{1}{p} \sum_{j=1}^{p} \delta_{\bar{X}_{j}^{\ast}}$, it follows by using Fubini's theorem that
\begin{eqnarray}
\E \left[ d_W^{2}(\neb, \nu_{0}) \right] & = & \E \left[ \int_{0}^{1} \left(F^{-}_{\neb}(\alpha) -  F_{0}^{-}(\alpha)    \right)^{2} d\alpha \right] = \E \left[   \sum_{j=1}^{p} \int_{(j-1)/p}^{j/p} \left(\bar{X}_{j}^{\ast} - F_{0}^{-}(\alpha) \right)^2 d\alpha \right] \nonumber \\
& = &\sum_{j=1}^{p} \int_{(j-1)/p}^{j/p}   \E \left[\bar{X}_{j}^{\ast} - F_{0}^{-}(\alpha) \right]^2 d\alpha \nonumber \\
& = & \sum_{j=1}^{p} \int_{(j-1)/p}^{j/p}   \E  \left[\bar{X}_{j}^{\ast}  -\E\left[ \bar{X}_{j}^{\ast} \right] \right]^2    +   \left(\E \left[ \bar{X}_{j}^{\ast} \right]   - F_{0}^{-}(\alpha) \right)^2  d\alpha  \nonumber \\
& = & \frac{1}{p} \sum_{j=1}^{p} \var \left( \bar{X}_{j}^{\ast}  \right) + \sum_{j=1}^{p} \int_{(j-1)/p}^{j/p}   \left(\E\left[ \bar{X}_{j}^{\ast} \right] - F_{0}^{-}(\alpha) \right)^2 d\alpha. \label{eq:eg1}
\end{eqnarray}
From Lemma \ref{lem:A}, one has that $\E\left[ \bar{X}_{j}^{\ast} \right] = \E \left[ Y_{j}^{\ast}  \right]$. Therefore, by combining  \eqref{eq:eg1} with  \eqref{eq:W2order}, we obtain that
\begin{eqnarray}
\E \left[ d_W^{2}(\neb, \nu_{0}) \right]  & = &  \frac{1}{p} \sum_{j=1}^{p} \var \left( \bar{X}_{j}^{\ast}  \right) + \sum_{j=1}^{p} \int_{(j-1)/p}^{j/p}   \left( \E \left[ Y_{j}^{\ast}  \right]- F_{0}^{-}(\alpha) \right)^2 d\alpha \nonumber \\
& = & \frac{1}{p} \sum_{j=1}^{p}\left( \var \left( \bar{X}_{j}^{\ast}  \right) - \var \left( Y_{j}^{\ast}  \right) \right)+ \E \left[ d_W^{2}(\bmu_{p}, \nu_{0}) \right]  \nonumber \\
& = & \frac{1}{n}  \left(  \int_{0}^{1}  \var \left(  \bF^{-}(\alpha)   \right)  d\alpha   \right) +  \frac{1-n}{p n} \sum_{j=1}^{p}  \var \left( Y_{j}^{\ast}  \right) + \E \left[ d_W^{2}(\bmu_{p}, \nu_{0}) \right] \\
& = & \frac{1}{n}  \int_{0}^{1}  \var\left(  \bF^{-}  (\alpha) \right)d\alpha +  \frac{1}{p n} \sum_{j=1}^{p}  \var \left( Y_{j}^{\ast}  \right) + \sum_{j=1}^{p} \int_{(j-1)/p}^{j/p}   \left( \E \left[ Y_{j}^{\ast}  \right]- F_{0}^{-}(\alpha) \right)^2 d\alpha , \nonumber
% & = &  \frac{1}{p} \sum_{j=1}^{p} \E \left(  |\bar{X}_{j}^{\ast}|^2  \right) - \E \left( |Y_{j}^{\ast}|^2  \right) + \E \left( d_W^{2}(\bmu_{p}, \nu_{0}) \right),  \nonumber \\
%  & \leq &  \frac{1}{n}  \int_{0}^{1}  \var\left(  \bF^{-}  (\alpha) \right)d\alpha+ \E \left[ d_W^{2}(\bmu_{p}, \nu_{0}) \right],
\end{eqnarray}
where the last equalities also follow from Lemma \ref{lem:A} and equality \eqref{eq:W2order}, which completes the proof of Theorem \ref{theo:ratehatbnu}.

\subsection{Proof of Theorem \ref{theo:ratehatbnupdiff}}

We recall that $\bnu_{n}^{\oplus}$ denotes the measure with quantile function given by equation \eqref{eq:quantoplus}. By the triangle inequality, we have that
\begin{equation}
d_W(\nebpbar, \nu_{0}) \leq d_W(\nebpbar, \bnu_{n}^{\oplus}) + d_W(\bnu_{n}^{\oplus}, \nu_{0}). \label{eq:triangle0}
\end{equation}
Thanks to Definition \ref{defi:wdist} of the Wasserstein distance, % and the fact that $\Omega$ is assumed to be a compact interval,
it follows by Fubini's theorem that
$$
\E \left[  d^2_W(\bnu_{n}^{\oplus}, \nu_{0}) \right] = \int_{0}^{1} \E \left[ \bar{\bF}_{n}^{-}(\alpha) - F_{0}^{-}(\alpha) \right]^2 d\alpha =  \int_{0}^{1} \E \left[  \frac{1}{n} \sum_{i=1}^{n} \bF^{-}_{i}(\alpha) - F_{0}^{-}(\alpha)\right]^2 d\alpha .
$$
By Assumption \ref{A3}, one has that $\E \left[ \bF^{-}_{i}(\alpha) \right] =  F_{0}^{-}(\alpha)$ for any $1 \leq i \leq n$, and thus by independence of the random variables $\bF^{-}_{i}(\alpha)$ one obtains that
\begin{equation}
\E \left[  d^2_W(\bnu_{n}^{\oplus}, \nu_{0}) \right] = \frac{1}{n}  \int_{0}^{1} \var\left(  \bF^{-}  (\alpha) \right)d\alpha.  \label{eq:var0}
\end{equation}
Hence, by \eqref{eq:var0}  and the inequality $\E \left[  d_W(\bnu_{n}^{\oplus}, \nu_{0}) \right] \leq \sqrt{ \E \left[  d^2_W(\bnu_{n}^{\oplus}, \nu_{0}) \right] } $, one obtains that
\begin{equation}
\E \left[  d_W(\bnu_{n}^{\oplus}, \nu_{0}) \right]  \leq n^{-1/2} \sqrt{  \int_{0}^{1} \var\left(  \bF^{-}  (\alpha) \right)d\alpha }.\label{eq:Eoplus0}
\end{equation}
Now, let us remark that
\begin{eqnarray*}
d_W(\nebpbar, \bnu_{n}^{\oplus}) & = &  \left\| \frac{1}{n} \sum_{i=1}^{n} F^{-}_{\tilde{\bnu}_{i}}  - \frac{1}{n} \sum_{i=1}^{n} \bF^{-}_{i} \right\| \leq  \frac{1}{n} \sum_{i=1}^{n} \left\| F^{-}_{\tilde{\bnu}_{i}}  - \bF^{-}_{i} \right\| ,
\end{eqnarray*}
where $F^{-}_{\tilde{\bnu}_{i}}$ denotes the quantile function of the measure $\tilde{\bnu}_{i} = \frac{1}{p_{i}} \sum_{j=1}^{p_{i}} \delta_{X_{i,j}}$ for each $1 \leq i \leq n$, and $\| \cdot \|$ denotes the usual norm in $L^{2}([0,1],dx)$. Hence, the above inequality leads to the following upper bound
\begin{equation}
\E \left[ d_W(\nebpbar, \bnu_{n}^{\oplus}) \right] \leq \frac{1}{n} \sum_{i=1}^{n} \sqrt{ \E \left[ d^2_W(\tilde{\bnu}_{i}, \bnu_{i}) \right] }. \label{eq:kernJ200}
\end{equation}
Therefore, Theorem \ref{theo:ratehatbnupdiff} follows from inequality \eqref{eq:triangle0} combined with \eqref{eq:Eoplus0} and \eqref{eq:kernJ200} which completes its proof.

\subsection{Proof of Theorem \ref{theo:ratehatbnuh}}

The proof follows the same lines than those of the proof of Theorem \ref{theo:ratehatbnupdiff}. By the triangle inequality, we have that
\begin{equation}
d_W(\seb, \nu_{0}) \leq d_W(\seb, \bnu_{n}^{\oplus}) + d_W(\bnu_{n}^{\oplus}, \nu_{0}). \label{eq:triangle}
\end{equation}
where  $\bnu_{n}^{\oplus}$ is the measure with quantile function given by equation \eqref{eq:quantoplus}. The expectation of the second term in the right-hand size of inequality \eqref{eq:triangle} is controlled by inequality \eqref{eq:Eoplus0}. Then, to control the first term, it suffices to remark that
\begin{eqnarray*}
d_W(\seb, \bnu_{n}^{\oplus}) & = &  \left\| \frac{1}{n} \sum_{i=1}^{n} F^{-}_{\hat{\bnu}^{h_{i}}_{i}}  - \frac{1}{n} \sum_{i=1}^{n} \bF^{-}_{i} \right\|
\end{eqnarray*}
where $F^{-}_{\hat{\bnu}^{h_{i}}_{i}}$ denotes the quantile function of the measure $\hat{\bnu}^{h_{i}}_{i}$ defined in \eqref{eq:kernsmooth}, and $\| \cdot \|$ denotes the usual norm in $L^{2}([0,1],dx)$. Therefore, one has that
\begin{eqnarray*}
d_W(\seb, \bnu_{n}^{\oplus}) & \leq & \frac{1}{n} \sum_{i=1}^{n}   \left\| F^{-}_{\hat{\bnu}^{h_{i}}_{i}}  - \bF^{-}_{i} \right\| \\
& = & \frac{1}{n} \sum_{i=1}^{n} d_W(\hat{\bnu}^{h_{i}}_{i}, \bnu_{i}) \leq    \frac{1}{n} \sum_{i=1}^{n} d_W(\hat{\bnu}^{h_{i}}_{i}, \tilde{\bnu}_{i})  +  \frac{1}{n} \sum_{i=1}^{n} d_W(\tilde{\bnu}_{i}, \bnu_{i}),
\end{eqnarray*}
where $\tilde{\bnu}_{i} = \frac{1}{p_{i}} \sum_{j=1}^{p_{i}} \delta_{X_{i,j}}$ for each $1 \leq i \leq n$. Hence, the above inequalities lead to the following upper bound
$$
\E \left[ d_W(\seb, \bnu_{n}^{\oplus}) \right] \leq \frac{1}{n} \sum_{i=1}^{n} \sqrt{ \E \left[ d^2_W(\hat{\bnu}^{h_{i}}_{i}, \tilde{\bnu}_{i}) \right] }  +  \frac{1}{n} \sum_{i=1}^{n} \sqrt{ \E \left[ d^2_W(\tilde{\bnu}_{i}, \bnu_{i}) \right] }.
$$
Finally, by applying Lemma \ref{lemma:kernsmooth} and inequality \eqref{eq:upBL}, we  obtain that
\begin{equation}
\E \left[ d_W(\seb, \bnu_{n}^{\oplus}) \right] \leq C_{\psi}^{1/2} \left( \frac{1}{n} \sum_{i=1}^{n} h_{i} \right) + \sqrt{2     \E \left[ J_{2}(\bnu) \right] }\left( \frac{1}{n} \sum_{i=1}^{n} p_{i}^{-1/2} \right). \label{eq:kernJ2}
\end{equation}
%Note that the above upper bound is necessarily finite thanks to the assumptions of Theorem \ref{theo:ratehatbnuh}. 
Therefore, Theorem \ref{theo:ratehatbnuh} follows from inequality \eqref{eq:triangle} combined with \eqref{eq:Eoplus0} and \eqref{eq:kernJ2} which completes its proof.

\subsection{Proof of Theorem \ref{theo:lowerbound}}

Let $A > 0$ and $\sigma > 0$. To derive Theorem \ref{theo:lowerbound}, we follow the classical scheme in nonparametric statistics to obtain optimal rates of convergence  (see Chapter 2 in \cite{MR2724359}). To this end, we introduce  appropriate random measures in $\WS$, satisfying the deformable model defined in Section \ref{sec:deform}, that will serve as the basic hypotheses to obtain a lower bound.

Let $m^{(1)}$ and $m^{(2)}$ be two real numbers such that
\begin{equation}
|m^{(1)} - m^{(2)}| = 2 C n^{-1/2}, \label{eq:m1m2}
\end{equation}
where $C$ is a positive constant to be specified later on. For $k = 1,2$, we let $\ba^{(k)}$ be independent Gaussian random variables with $\E \left[\ba^{(k)}\right] = m^{(k)}$ and $\var(\ba^{(k)}) = \gamma^2$ with $\gamma = \min(A^{1/2},\sigma)$. We also let $\HH^{(k)}$ denote the hypothesis that the data are sampled according the following deformable model:
\begin{equation}
X_{i,j}^{(k)} = \ba^{(k)}_{i} + Z_{i,j}^{(k)},  \quad 1 \leq i \leq n, \; 1 \leq j \leq p_{i}, \label{model:lowerbound}
\end{equation}
where $\ba^{(k)}_{1},\ldots,\ba^{(k)}_{n}$ are independent copies of $\ba^{(k)}$, and the $Z_{i,j}^{(k)}$'s are iid random variables sampled from the  Gaussian distribution with zero mean and variance $\gamma^2$, that are independent of the $\ba^{(k)}_{i}$'s. If we let $\bX_{i}^{(k)}$ be the random vector in $\R^{p_{i}}$ whose component are the random variables $(X_{i,j}^{(k)})_{1 \leq j \leq p_{i}}$, then the deformable model \eqref{model:lowerbound} corresponds to the assumption that $\bX_{1}^{(k)}, \ldots, \bX_{n}^{(k)}$ are independent random vectors, such that $\bX_{i}^{(k)}$ is a Gaussian vector with
\begin{equation}
\E \left[ \bX_{i}^{(k)} \right] = m^{(k)} \be_{i} \quad \mbox{ and } \quad \var \left(  \bX_{i}^{(k)}  \right) = \gamma^2 \left( \be_{i} \be_{i}^{t} + \bI_{i} \right), \label{eq:meancov}
\end{equation}
where $\be_{i}$ is the vector in $\R^{p_{i}}$ with all entries equal to one,  the notation  $\var \left( \bX \right)$ denotes the covariance matrix of a random vector $\bX$, and $\bI_{i}$ is the identity $p_{i} \times p_{i}$ matrix. For each $k=1,2$, if we denote by $\bnu_{i}^{(k)}$ the  measure from which $(X_{i,j}^{(k)})_{1 \leq j \leq p_{i}}$ are sampled, it follows, from model \eqref{model:lowerbound}, that $\bnu_{1}^{(k)},\ldots,\bnu_{n}^{(k)}$ are independent copies of the random measure $\bnu^{(k)}$ with density $\frac{1}{\gamma} \phi_{0}\left( \frac{x-\ba^{(k)}}{\gamma} \right), x \in \R$, where $\phi_{0}$ is the density of the standard Gaussian distribution. It can be easily checked that the barycenter $\nu_{0}^{(k)}$ in $W_{2}(\R)$ of the random measure $\bnu^{(k)}$ is the Gaussian distribution with mean $m^{(k)}$ and variance $\gamma^2$, and that
\begin{equation}
d_W(\nu_{0}^{(1)}, \nu_{0}^{(2)}) = |m^{(1)} - m^{(2)}| = 2 C n^{-1/2}. \label{eq:dW2}
\end{equation}
Hence, $\nu_{0}^{(k)}$ belongs to the class of distributions  $\FF(\R, A)$ introduced in Definition \ref{defin:classF}, for $k=1,2$.
Moreover, since $F_{\bnu^{(k)}}^{-}(\alpha) = \Phi_{0}^{-}(\alpha) + \ba^{(k)}, \; t \in [0,1]$ where $\Phi_{0}^{-}$ is the quantile function of the standard Gaussian distribution, it follows that
$$
\int_{0}^{1} \var \left( F_{\bnu^{(k)}}^{-}(\alpha) \right) d\alpha = \int_{0}^{1} \var \left(  \ba^{(k)} \right) d\alpha = \gamma^2 \leq \sigma^2.
$$
Therefore, the random measure $\bnu^{(k)}$ belongs to the class of distributions  $\DD(\R,\nu_{0}^{(k)}, \sigma^2)$ introduced in Definition \ref{defin:classD}, for $k=1,2$.

Then, for $k=1,2$, we let $\P^{(k)}$ be the probability measure of the data in model \eqref{model:lowerbound} under the hypothesis $\HH^{(k)}$. From our remark above, one has that $\P^{(k)}$ is the product of $n$ Gaussian measures $\P^{(k)}_{i}$ on $\R^{p_{i}}$ with mean and covariance given by \eqref{eq:meancov} for $1 \leq i \leq n$. Hence, the Kullback divergence  $K\left(\P^{(1)}, \P^{(2)} \right)$ between  $\P^{(1)}$ and $\P^{(2)}$ can be decomposed as follows
\begin{eqnarray}
K\left(\P^{(1)}, \P^{(2)} \right) & = & \sum_{i=1}^{n} K\left(\P^{(1)}_{i}, \P^{(2)}_{i} \right) \nonumber \\
& = & \frac{1}{2 \gamma^2} |m^{(1)} - m^{(2)}|^2 \sum_{i=1}^{n} \be_{i}^{t}  \left( \be_{i}\be_{i}^{t}  + \bI_{i} \right)^{-1} \be_{i} \nonumber \\
& = &  \frac{1}{2 \gamma^2} |m^{(1)} - m^{(2)}|^2 \sum_{i=1}^{n} \frac{p_{i}}{p_{i}+1} \leq \frac{n}{2 \gamma^2} |m^{(1)} - m^{(2)}|^2 \nonumber \\
& \leq & 2 C^{2} \max(A^{-1},\sigma^{-2}), \label{eq:boundKull}
\end{eqnarray}
where the last inequality follows from \eqref{eq:m1m2} and the fact that $\gamma^2 = \min(A,\sigma^{2})$. 

To conclude the proof, we finally follow the arguments from Section 2.2 in \cite{MR2724359} on a reduction scheme to a finite number $M$ of hypotheses (here $M=2$). First, thanks to Markov's inequality, one has that
$$
\inf_{\hat{\bnu}} \sup_{\nu_{0} \in \FF(\R,A) } \sup_{ \bnu \in \DD(\R,\nu_{0}, \sigma^2) } \E \left[ n^{1/2} d_W(\hat{\bnu}, \nu_{0}) \right] \geq C \inf_{\hat{\bnu}} \sup_{\nu_{0} \in \FF(\R,A) } \sup_{ \bnu \in \DD(\R,\nu_{0}, \sigma^2) } \P \left(  d_W(\hat{\bnu}, \nu_{0}) \geq C n^{-1/2}  \right),
$$
and thus, the following lower bound holds
\begin{equation}
\inf_{\hat{\bnu}} \sup_{\nu_{0} \in \FF(\R,A) } \sup_{ \bnu \in \DD(\R,\nu_{0}, \sigma^2) } \E \left[ n^{1/2} d_W(\hat{\bnu}, \nu_{0}) \right] \geq C \inf_{\hat{\bnu}} \max_{k \in \{1,2\}}  \P^{(k)} \left( d_W(\hat{\bnu}, \nu_{0}^{(k)})  \geq C n^{-1/2} \right), \label{eq:lower1}
\end{equation}
where  $\P^{(k)}$ denotes the probability measure of the data in model \eqref{model:lowerbound} under the hypothesis $\HH^{(k)}$  for $k=1,2$.  Now,  thanks to equality \eqref{eq:dW2}, the two hypotheses $\HH^{(1)}$ and $\HH^{(2)}$  are  $2$s-separated in the sense of condition (2.7) in  \cite{MR2724359}  (with $s = C n^{-1/2} $). Hence, by inequality (2.9) in \cite{MR2724359}, one has that
\begin{equation}
\inf_{\hat{\bnu}} \max_{k \in \{1,2\}}  \P^{(k)} \left( d_W(\hat{\bnu}, \nu_{0}^{(k)})  \geq C n^{-1/2}  \right) \geq p_{e,1},  \label{eq:lower2}
\end{equation}
where $p_{e,1}$ is defined by equation (2.10) in  \cite{MR2724359}. Then, by the upper bound \eqref{eq:boundKull} on the Kullback divergence between  $\P^{(1)}$ and $\P^{(2)}$, we can combine the Kullback version of Theorem 2.2 in \cite{MR2724359} with inequalities \eqref{eq:lower1} and \eqref{eq:lower2} to obtain that
$$
\inf_{\hat{\bnu}} \sup_{\nu_{0} \in \FF(\R,A) } \sup_{ \bnu \in \DD(\R,\nu_{0}, \sigma^2) } \E \left[ n^{1/2} d_W(\hat{\bnu}, \nu_{0}) \right] \geq C p_{e,1} \geq C \max \left( \frac{1}{4} \exp(-\alpha), \frac{1-\sqrt{\alpha / 2}}{2}\right)  
$$
with $\alpha = 2 C^{2} \max(A^{-1},\sigma^{-2})$. Therefore, taking $C =  \min(A^{1/2},\sigma)$ completes the proof of Theorem \ref{theo:lowerbound}.

\bibliographystyle{alpha}
\bibliography{RatebarycenterW2}

\begin{thebibliography}{BGKL15}

\bibitem[AC11]{agueh2011barycenters}
M.~Agueh and G.~Carlier.
\newblock Barycenters in the {W}asserstein space.
\newblock {\em SIAM Journal on Mathematical Analysis}, 43(2):904--924, 2011.

\bibitem[BGKL15]{BGKL15}
J.~Bigot, R.~Gouet, T.~Klein, and A.~Lopez.
\newblock Geodesic {PCA} in the {W}asserstein space by {C}onvex {PCA}.
\newblock {\em Annales de l'Institut Henri Poincar\'e B: Probability and
  Statistics}, To be published, 2015.

\bibitem[BIAS03]{Bolstad2003}
B.~M. Bolstad, R.~A. Irizarry, M.~Astrand, and T.~P. Speed.
\newblock A comparison of normalization methods for high density
  oligonucleotide array data based on variance and bias.
\newblock {\em Bioinformatics}, 19(2):185--193, January 2003.

\bibitem[BK16]{BK16}
J.~Bigot and T.~Klein.
\newblock Characterization of barycenters in the {W}asserstein space by
  averaging optimal transport maps.
\newblock {\em Preprint, https://hal.archives-ouvertes.fr/hal-00763668v5},
  2016.

\bibitem[BL14]{W1}
S.~Bobkov and M.~Ledoux.
\newblock {\em One-dimensional empirical measures, order statistics and
  Kantorovich transport distances}.
\newblock Book in preparation, 2014.
\newblock Available at
  {http://perso.math.univ-toulouse.fr/ledoux/files/2013/11/Order.statistics.10.pdf}.

\bibitem[BLGL15]{MR3338645}
E.~Boissard, T.~Le~Gouic, and J.-M. Loubes.
\newblock Distribution's template estimate with {W}asserstein metrics.
\newblock {\em Bernoulli}, 21(2):740--759, 2015.

\bibitem[dBGU05]{BGU05}
E.~del Barrio, E.~Gin{\'e}, and F.~Utzet.
\newblock Asymptotics for {$L_2$} functionals of the empirical quantile
  process, with applications to tests of fit based on weighted {W}asserstein
  distances.
\newblock {\em Bernoulli}, 11(1):131--189, 2005.

\bibitem[Del11]{MR2736564}
P.~Delicado.
\newblock Dimensionality reduction when data are density functions.
\newblock {\em Comput. Statist. Data Anal.}, 55(1):401--420, 2011.

\bibitem[Fr{\'e}48]{fre}
M.~Fr{\'e}chet.
\newblock Les {\'e}l{\'e}ments al{\'e}atoires de nature quelconque dans un
  espace distanci{\'e}.
\newblock {\em Ann. Inst. H.Poincar{\'e}, Sect. B, Prob. et Stat.},
  10:235--310, 1948.

\bibitem[KU01]{MR1946423}
A.~Kneip and K.~J. Utikal.
\newblock Inference for density families using functional principal component
  analysis.
\newblock {\em J. Amer. Statist. Assoc.}, 96(454):519--542, 2001.
\newblock With comments and a rejoinder by the authors.

\bibitem[LH10]{li2010}
Y.~Li and T.~Hsing.
\newblock Uniform convergence rates for nonparametric regression and principal
  component analysis in functional/longitudinal data.
\newblock {\em Annals of Statistics}, 38(6):3321--3351, 12 2010.

\bibitem[PM15]{PetersenMuller}
K.~Petersen and H.-G. M\"{u}ller.
\newblock Functional data analysis for density functions by transformation to a
  {H}ilbert space.
\newblock {\em Annals of Statistics}, To be published, 2015.

\bibitem[PZ16]{Pana15}
V.M. Panaretos and Y.~Zemel.
\newblock Amplitude and phase variation of point processes.
\newblock {\em Annals of Statistics}, 44(2):771--812, 2016.

\bibitem[RL01]{ramli}
J.O. Ramsay and X.~Li.
\newblock Curve registration.
\newblock {\em Journal of the Royal Statistical Society (B)}, 63:243--259,
  2001.

\bibitem[Tsy09]{MR2724359}
A.~B. Tsybakov.
\newblock {\em Introduction to nonparametric estimation}.
\newblock Springer Series in Statistics. Springer, New York, 2009.
\newblock Revised and extended from the 2004 French original, Translated by
  Vladimir Zaiats.

\bibitem[Vil03]{villani-topics}
C.~Villani.
\newblock {\em {Topics in Optimal Transportation}}, volume~58 of {\em Graduate
  Studies in Mathematics}.
\newblock American Mathematical Society, 2003.

\bibitem[WG97]{wanggas}
K.~Wang and T.~Gasser.
\newblock Alignment of curves by dynamic time warping.
\newblock {\em Annals of Statistics}, 25(3):1251--1276, 1997.

\bibitem[WS11]{Srivastava}
W.~Wu and A.~Srivastava.
\newblock {An information-geometric framework for statistical inferences in the
  neural spike train space}.
\newblock {\em Journal of Computational Neuroscience}, 31(3):725--748, November
  2011.

\bibitem[ZM11]{ZhangMuller}
Z.~Zhang and H.-G. M\"{u}ller.
\newblock Functional density synchronization.
\newblock {\em Computational Statistics \& Data Analysis}, 55(7):2234--2249,
  2011.

\end{thebibliography}

\end{document}